\documentclass[12pt,
computermodern]{amsart}

\usepackage{amssymb,amsfonts,amsmath}
\usepackage[all,arc]{xy}
\usepackage{enumerate}
\usepackage{mathrsfs}
\usepackage{fullpage}
\usepackage{xspace}
\usepackage[margin=1.0in]{geometry}
\usepackage{tcolorbox}
\usepackage{tikz-cd}
\usepackage{color}
\usepackage{aliascnt}
\usepackage[foot]{amsaddr}
\usepackage{hyperref}

\newtheorem{thm}{Theorem}[section]

\newaliascnt{theo}{thm}

\aliascntresetthe{theo}

\newaliascnt{cor}{thm}
\newtheorem{cor}[cor]{Corollary}
\aliascntresetthe{cor}

\newaliascnt{prop}{thm}
\newtheorem{prop}[prop]{Proposition}
\aliascntresetthe{prop}

\newaliascnt{lem}{thm}
\newtheorem{lem}[lem]{Lemma}
\aliascntresetthe{lem}

\newaliascnt{conj}{thm}

\aliascntresetthe{conj}

\newaliascnt{que}{thm}

\aliascntresetthe{que}

\newaliascnt{ass}{thm}

\aliascntresetthe{ass}

\newaliascnt{defn}{thm}
\newtheorem{defn}[defn]{Definition}
\aliascntresetthe{defn}

\theoremstyle{remark}
\newaliascnt{rem}{thm}
\newtheorem{rem}[rem]{Remark}
\aliascntresetthe{rem}

\theoremstyle{definition}

\newtheorem{exmp}[thm]{Example}

\newtheorem{notation}[thm]{Notation}

\newcommand{\C}{\mathbb{C}\xspace}

\newcommand{\Z}{\mathbb{Z}\xspace}
\newcommand{\Q}{\mathbb{Q}\xspace}
\newcommand{\G}{\mathbb{G}\xspace}

\newcommand{\ch}{\text{CH}}
\DeclareMathOperator{\Spec}{Spec}
\DeclareMathOperator{\res}{res}
\DeclareMathOperator{\Tr}{Tr}
\DeclareMathOperator{\ord}{ord}

\DeclareMathOperator{\dv}{div}
\DeclareMathOperator{\alb}{alb}

\DeclareMathOperator{\Pic}{Pic}

\DeclareMathOperator{\Alb}{Alb}

\makeatletter
\let\c@equation\c@theo
\makeatother
\numberwithin{equation}{section}

\title{A Filtration of the Chow Group of Zero-Cycles for a Product of Curves and an Abelian Variety}
\author{Thomas Jaklitsch$^*$}\thanks{\noindent(*) Department of Mathematics, University of Virginia, 119 Kerchof Hall, 141 Cabell Dr., Charlottesville, VA, 22903, USA. Email: gvs3ka@virginia.edu}
\begin{document}

\maketitle
\begin{abstract}
    In this paper we define a descending filtration on the Chow group of zero cycles for varieties of the form $A \times C_1 \times \cdots \times C_d$ where $A$ is an abelian variety and each $C_i$ is a smooth projective curve. We give explicit generators and relations for the successive quotients of this filtration by showing that they can be described by Somekawa K-groups. This extends the work of Raskind and Spiess who proved this result for products of curves and Gazaki who proved this for abelian varieties. 
\end{abstract}

\section{Introduction}
Let $X$ be a smooth projective variety over a field $k$ such that $X(k) \ne \emptyset$, and let $\ch_0(X)$ denote the group of zero cycles modulo rational equivalence. There is a well-defined degree map 
\[
    \text{deg}: \ch_0(X) \to \Z
\]
and we denote the kernel by $A_0(X)$. When $X$ is a curve, $\ch_0(X)$ coincides with the Picard group $\text{Pic}(X)$, and the Abel-Jacobi map
\[
    A_0(X) \to \text{Jac}_X(k)
\]
is an isomorphism. In higher dimensions, however, the situation is more complicated. Like in the case of curves, there is an abelian variety $\Alb_X$, the Albanese variety of $X$, which is universal with respect to maps from $X$ to abelian varieties and a homomorphism 
\[
    \alb_X: A_0(X) \to \Alb_X(k).
\]
Unlike the Abel-Jacobi map for curves, however, this map is very far from being an isomorphism in general. In particular, the kernel of this map can be quite large. Thus, we now have a filtration 
\[
    \ker(\alb_X) \subset A_0(X) \subset \ch_0(X)
\]
where the first successive quotient is isomorphic to $\Z$ and the second is isomorphic to a subgroup of the $k$-points of an abelian variety. Over many fields of interest, the structure of the group of rational points on an Abelian variety is well understood, so the main obstruction to our understanding of $\ch_0(X)$ is in the kernel of the Albanese map $\alb_X$. 

The structure of the Albanese kernel is very sensitive to changes in the base field. For example, when $k = \C$ and $X$ is a surface with positive geometric genus, $\ker(\alb_X)$ is extremely large \cite{bloch1975k2}, \cite{mumford1969rational}, and it is not parameterized by the points of a variety. On the other hand, famous conjectures of Beilinson \cite{beilinson2006height} (conjecture 5.0) and Bloch \cite{bloch1984algebraic} (page 94) predict that $\ker(\alb_X)$ is finite for every smooth projective variety $X$ defined over $\Q$. The variance of the structure of the Albanese kernel as the base field changes illustrates the difficulty of studying this group. Indeed, there are very few known examples of the Beilinson-Bloch conjecture for surfaces defined over $\Q$. 

One way to make the problem of studying the Albanese kernel more tractable is by extending the natural filtration on $\ch_0(X)$ and identifying the successive quotients with groups which are more computable. In this paper, we construct a descending integral filtration 
\[
    \ldots F^r\ch_0(X) \subset \ldots \subset F^1\ch_0(X) \subset F^0\ch_0(X) = \ch_0(X)
\]
on the Chow group of zero-cycles when $X \cong C_1\times \ldots C_d\times A$ where $A$ is an abelian variety, and $C_1, \ldots, C_d$ are smooth, projective, geometrically integral curves with a $k$-rational point. We relate the successive quotients of this filtration to certain $K$-groups as studied by Somekawa in \cite{somekawa1990milnor}.

More precisely, Somekawa defines the groups $K(k; G_1,\ldots, G_r)$ for semi-abelian varieties $G_1, \ldots, G_r$. These groups are a generalization of Milnor $K$-groups in the sense that when $G_i = \G_m$ for all $i$, $K(k; G_1, \ldots, G_r) \cong K^M_r(k)$. We define a natural quotient of the Somekawa $K$-groups and prove they are isomorphic to the successive quotients of our filtration up to bounded torsion. Our main result is the following. 
\begin{thm}\label{thm: main}
    Let $A$ be an abelian variety and let $C_1, \ldots, C_d$ be smooth projective geometrically irreducible curves such that $C(k) \ne \emptyset$. Then there exists an isomorphism
    \[
        \mathbb{Z}[1/r!]\otimes \frac{F^r(C_1\times \ldots C_d\times A)}{F^{r+1}(C_1\times \ldots C_d\times A)} \xrightarrow{\Z[1/r!]\otimes \Phi_r} \mathbb{Z}[1/r!]\otimes S_r(k; \underline{J_1}\times \cdots \times \underline{J_d}\times A).
    \]
\end{thm}
Here $S_r(k; \underline{J_1}\times \cdots \times \underline{J_d}\times A)$ is a quotient of a Somekawa $K$-group as described above (see Definition \ref{def: main K-group def}). The map $\Phi_r$ is defined in section \ref{section: Filtration}. Importantly, this filtration extends the natural filtration of $\ch_0(X)$ in the sense that $F^1\ch_0(X) = A_0(X)$ and $F^2\ch_0(X) = \ker(\alb_X)$. One of the main utilities of this is that the $K$-groups give explicit generators and relations for the successive quotients, which aids in the study of the Albanese kernel of $X$. To illustrate this point, we provide some computations in section \ref{section: Example computation} to show how this filtration might be used to study the Albanese kernel in a special case. Moreover, our filtration is constructed over any perfect field and is integral. These properties are essential for the filtration to be useful in studying questions of arithmetic interest. 

Such a filtration has been constructed in the case when $X$ is a product of curves by Raskind and Spiess in \cite{raskind2000milnor}. It has also been constructed in the case when $X$ is an abelian variety by Gazaki in \cite{gazaki2015filtration}. Kakinoki extends the results of Gazaki in \cite{kakinoki2020filtration} to construct a filtration for all Chow groups $\ch^{g+s}(A, s)$ where $A$ is an abelian variety, $g = \text{dim}(A)$, and $\ch^{g+s}(A,s)$ denote Bloch's higher Chow groups \cite{bloch1986algebraic}.

In fact, it is expected that similar filtrations of $\ch_0(X)$ should exist for any smooth projective variety $X$ over $k$. More precisely, a conjecture of Beilinson predicts that there is an abelian category $\mathcal{MM}_k$ (the category of mixed motives over $k$) and a spectral sequence 
\[
    E^{\nu,\mu}_2 = \text{Ext}^\nu_{\mathcal{MM}_k}(1, h^\mu(X)(m))\implies \ch^m(X, 2m-(\nu+\mu)).
\]
In general, this spectral sequence is expected to degenerate after tensoring with $\Q$, so we would obtain a filtration on $\ch_0(X)$ with the property that after tensoring with $\Q$ the $r$th successive quotient is isomorphic to the Ext group
\[
    \text{Ext}_{\mathcal{MM}_k}^r(1, h^{2d-r}(X)(d))
\]
where $d$ is the dimension of $X$ (for more information on this conjecture of Beilinson and its applications see \cite{jannsen1994motivic}). While this conjecture is far out of reach in general, the work of Raskind and Spiess in \cite{raskind2000milnor} provides strong evidence that the filtration they construct for a product of curves is the one arising from the above conjecture of Beilinson for the following reason. In \cite{somekawa1990milnor}, Somekawa predicts that the $K$-groups he studies in that paper are isomorphic to certain $\text{Ext}$ groups in the conjectural category of mixed motives. Assuming this conjecture, Raskind and Spiess show that their filtration has the correct motivic interpretation (see Remarks 2.4.2 of \cite{raskind2000milnor}). Given the fact that the category of mixed motives over $k$ has not been established, Somekawa's expectation cannot be verified. However, Kahn and Yamazaki in \cite{kahn2013voevodsky} prove "the closest approximation to Somekawa's conjecture" by showing that 
\[
    K(k; G_1, \ldots, G_r) \cong \text{Hom}_{\textbf{DM}^\text{eff}_{-}}(\Z, G_1[0]\otimes \cdots \otimes G_r[0])
\]
where $\textbf{DM}^\text{eff}_{-}$ is Voevodsky's triangulated category of effective motivic complexes\footnote{In fact they show this isomorphism more generally for homotopy invariant Nisnevich sheaves with transfers}. Our filtration is an extension of the one constructed by Raskind and Spiess and the successive quotients are described by quotients of Somekawa $K$-groups. Therefore, the remarks above suggest that our filtration is closely related to the one that would arise from Beilinson's conjecture.

Another indication that our filtration is the one predicted by Beilinson is given by the following vanishing theorem.
\begin{thm}\label{thm: vanishing}
    For $r > \text{dim}(C_1\times\cdots \times C_d \times A)$, 
    \[
        F^r\ch_0(C_1\times\cdots \times C_d\times A)\otimes \Q = 0. 
    \]
\end{thm}
\noindent Such a property is expected for motivic filtrations. Namely, the filtration should vanish rationally after the dimension of the variety. Thus, the above theorem provides more evidence that our filtration is the expected one coming from the conjectural category of mixed motives. Moreover, the rational finiteness of our filtration plays a major role in computing generators and relations of the Albanese kernel as demonstrated by the example in Section \ref{section: Example computation}.

\subsection{Outline}
This paper follows the structure of \cite{gazaki2015filtration}. In Section 2 we recall some definitions regarding Chow groups and Somekawa $K$-groups which are the main objects of interest. In section 3 for $X:= C_1\times\cdots \times C_d\times A$, we define the map 
\[
    \Phi_r: \ch_0(X) \to S_r(k; \underline{J_1}\times \cdots \times \underline{J_d}\times A)
\]
which appears in the statement of Theorem \ref{thm: main} and use it to construct the filtration $F^\bullet\ch_0(X)$. We also verify that our filtration extends the natural one of $\ch_0(X)$. In Section 4 we prove the main theorem by defining a map $\Psi_r$ from $S_r(k; \underline{J_1}\times \cdots \times \underline{J_d}\times A)$ to the successive quotients of the filtration and verifying that this is an isomorphism after tensoring with $\Z[1/r!]$. In Section 5 we prove Theorem \ref{thm: vanishing}, and in Section 6 we give an example of how the filtration can be used to study the Albanese kernel of $\ch_0(C\times A)$ where $C$ is a hyperelliptic curve and $A$ is its Jacobian. 

\subsection{Notation}\label{notation: cycle class notation}
Throughout this paper, $k$ is a perfect field. For a variety $X/k$, and a field extension $L/k$, we denote the base change $X_L := X \times_{\Spec k} \Spec L$. For a closed point $x \in X$, $k(x)$ will denote its residue field, and $[x]$ will denote its cycle class in the Chow group $\ch_0(X)$. If $L/k(x)$ is a finite extension, consider the map $\Spec L \to X\times_{\Spec k} \Spec L$ which is determined by $\Spec L \to \Spec k(x) \to X$ and $\text{id}: \Spec L \to \Spec L$. Then we define $[x]_L \in \ch_0(X_L)$ to be the class of the closed point associated to this map (For example, see Definition \ref{def:Psi-map}).

\subsection{Acknowledgments}
I would like to thank my thesis advisor Evangelia Gazaki for her unwavering support and guidance on this project. She introduced me to this problem and gave me invaluable feedback at every stage in the process.

\section{Definitions}
Let $k$ be a perfect field. Let $C_1,\ldots, C_d$ be smooth, projective, and geometrically integral curves over $k$, and let $A$ be an abelian variety over $k$. We assume $C_i(k) \ne \emptyset$ for $1\le i\le d$, and denote the Jacobian of $C_i$ by $J_i$. We fix points $p_i \in C_i(k)$ which determine embeddings $\iota_{p_i}: C_i \to J_i.$

\subsection{Somekawa K-groups}\label{section: Somekawa-Kgps}

We now recall the definition of the K-groups defined in \cite{somekawa1990milnor}. For more information on such groups see also \cite{gazaki2015filtration}. 

Let $A_1, \ldots, A_r$ be abelian varieties over $k$. Then we define the group
\[
    \tilde{K}(k; A_1,\ldots, A_r) := \left[\bigoplus_{k'/k}A_1(k') \otimes \cdots \otimes A_r(k')\right]/PF
\]
where the sum is taken over all finite extensions $k'/k$ and $PF$ is the subgroup generated by elements of the following form. If $L/E/k$ is a tower of finite extensions, $a_i \in A_i(L)$ for some $1 \le i \le n$, and $a_j \in A_j(E)$ for all $j\ne i$, then 
\[
    a_1\otimes \cdots \otimes \Tr_{L/E}(a_i)\otimes \cdots \otimes a_r - \res_{L/E}(a_1)\otimes \cdots \otimes a_i \otimes \cdots \otimes \res_{L/E}(a_r) \in PF.
\]
Relations of the above form will be referred to as the projection formula.
\begin{notation}
    We follow the convention of denoting a simple tensor $a_1\otimes \ldots \otimes a_n$ in the component $A_1(k') \otimes \ldots \otimes A_r(k')$ of the group $\tilde{K}(k; A_1,\ldots, A_r)$ by the symbol $\{a_1,\ldots, a_r\}_{k'/k}.$
\end{notation}

Now we define the Somekawa $K$-group associated to the abelian varieties $A_1, \ldots, A_r$ by 
\[
    K(k; A_1,\ldots, A_r) := \tilde{K}(k; A_1,\ldots, A_r)/WR
\]
where $WR$ is the subgroup generated by elements of the following form. Let $K$ be a function field in one variable over $k$. Let $f \in K^\times$ and $x_i \in A_i(K)$ for $i = 1,\ldots, n$. Then 
\[
    \sum_{v \text{ place of } K/k} \ord_v(f)(s_v^1(x_1)\otimes \cdots \otimes s^n_v(x_n)) \in R.
\]
Here the specialization maps $s_v^i: A_i(K) \to A_i(k_v)$ where $k_v$ is the residue field of the place $v$ are defined as follows. Let $\mathcal{O}_v$ be the ring of integers of $v$. Then the natural map $A_i(\mathcal{O}_v) \to A_i(K)$ is an isomorphism by the valuative criterion for properness. Then the specialization map is defined by 
\[
    \begin{tikzcd}
        A_i(K)\arrow{rd}[swap]{s^i_v} & A_i(\mathcal{O}_v) 
        \arrow{l}[swap]{\cong} \arrow{d} \\
        & A_i(k)
    \end{tikzcd}
\]

The relations coming from $WR$ are called Weil relations. 

In the case that $A_1 = A_2 = \ldots = A_r$, we will denote the $K$ group by 
\[
    K_r(k;A) := K(k;\underbrace{A, \ldots, A}_{r}).
\]
Let $\Sigma_r$ denote the group of permutations of $\{1,\ldots, r\}$. Then this group acts on $K_r(k;A)$ by
\[
    \sigma\cdot \{x_1,\ldots, x_r\}_{k'/k} = \{x_{\sigma(1)}, \ldots, x_{\sigma(r)}\}_{k'/k}.
\]
We define 
\[
    S_r(k;A) := K_r(k;A)_{\Sigma_r}.
\]
Thus, the symbols $\{a_1,\ldots, a_r\}_{k'/k}$ in $S_r(k;A)$ are symmetric.
\\ Given abelian varieties $A_1,\ldots, A_r$, there exists an isomorphism (\cite{gazaki2022filtrations}, Corollary 2.14)
\[
    g_r: S_r(k; A_1\times \ldots \times A_d) \xrightarrow{\cong} \bigoplus_{1\le i_1\le \ldots \le i_r \le d} S_r(k; A_{i_1},\ldots, A_{i_r}).
\]
In particular, if $A$ is an abelian variety and $J_i$ is the Jacobian of $C_i$
\[
    g_r: S_r(k; J_1\times \ldots \times J_d \times A) \xrightarrow{\cong} \bigoplus_{0 \le t \le r}\bigoplus_{1\le i_1\le \ldots \le i_{r-t} \le d} S_r(k; J_{i_1}, \ldots, J_{i_{r-t}}, \underbrace{A, \ldots, A}_{t})
\]

\begin{defn}\label{def: main K-group def}
Define the $K$-group
\[
    S_r(k; \underline{J_1}\times \ldots \times \underline{J_d} \times A) := g_r^{-1}\left(\bigoplus_{0 \le t \le r}\bigoplus_{1\le i_1< \ldots < i_{r-t} \le d} S_r(k; J_{i_1}, \ldots, J_{i_{r-t}}, \underbrace{A, \ldots, A}_{t})\right).
\]
That is, the group $S_r(k; \underline{J_1}\times \ldots \times \underline{J_d} \times A)$ is the subgroup of $S_r(k; J_1\times \ldots \times J_d \times A)$ composed of the direct summands which have at most one copy of each of the Jacobians. In the case when $r = 0$, we define $S_0(k; \underline{J_1}\times \ldots \times \underline{J_d}\times A) := \mathbb{Z}$.
\end{defn}

An important feature of the Somekawa $K$-groups that we will make use of is that they are functorial with respect to field extensions. Suppose that $L/k$ is a finite extension. Then the trace map is defined by 
\[
    \Tr_{L/k} : K(L; A_1\times_k L, \ldots, A_n\times_k L) \to K(k; A_1, \ldots, A_n)
\]
\[
    \{a_1, \ldots, a_n\}_{E/L} \mapsto \{\pi\circ a_1, \ldots, \pi\circ a_n\}_{E/k}
\]
where $\pi : A\times_k L \to A$ is the natural projection map.
Now let $L/k$ be an arbitrary extension. Then the restriction map
\[
    \res_{L/k}:K(k; A_1,\ldots, A_n) \to K(L; A_1\times_k L, \ldots, A_n\times_k L)
\]
is defined as follows. Let $E/k$ be a finite extension. Then since $k$ is perfect, $L\otimes_{k} E \cong \prod_{i=1}^m L_i$ where $L_i/L$ is a finite extension. Then we define
\[
    \res_{L/k}(\{a_1,\ldots, a_n\}_{E/k}) = \sum_{i=1}^m\{\res_{L_i/E}(a_1),\ldots, \res_{L_i/E}(a_n)\}_{L_i/L}.
\]
Note that if $L/k$ is a finite extension, and $\{a_1,\ldots, a_n\}_{k'/k} \in K(k; A_1,\ldots, A_n)$, then 
\[
    \Tr_{L/k}(\res_{L/k}(\{a_1,\ldots, a_n\}_{k'/k})) = [L:k]\{a_1,\ldots, a_n\}_{k'/k}.
\]
\subsection{Chow Groups}
Recall that given a smooth projective variety $Y$, the group $\ch^*(Y)$ becomes a graded ring under the intersection product, which we will denote by $\cdot$. Given a morphism $f: Y\to Z$ we have a well-defined pushforward
\[
    f_*: \ch_n(Y) \to \ch_n(Z)
\]
if $f$ is proper, and a well-defined pullback
\[
    f^*: \ch^n(Y) \to \ch^n(Z)
\]
if $f$ is flat. 

In the case of an abelian variety $A$, $\ch_*(A)$ also becomes a graded ring under the Pontryagin product, which we will denote by $\odot$. This is defined as 
\[
    \odot: \ch_r(A) \otimes \ch_s(A) \to \ch_{r+s}(A)
\]
\[
    \alpha \odot \beta := m_*(\alpha \times \beta)
\]
where $m: A\times A \to A$ is the multiplication map on the abelian variety. 

If $Y$ is a smooth projective variety over $k$ with a $k$-rational point $P$, then there exists an abelian variety $\Alb_Y$ and a map $\varphi: Y \to \Alb_Y$ which sends $P$ to $0$ and is universal with respect to pointed maps from $Y$ to abelian varieties. That is, for every abelian variety $A$ and morphism $\psi: Y \to A$ such that $\psi(P) = 0$ there exists a unique map $f: \Alb_Y \to A$ making the following diagram commute
\[
    \begin{tikzcd}
        Y \arrow{r}{\varphi} \arrow{rd}{\psi} & \Alb_Y \arrow{d}{f} \\
        & A.
    \end{tikzcd}
\]
Moreover, there exists a group homomorphism $\alb_Y: A_0(Y) \to \Alb_Y(k)$ which is independent of the choice of the rational point $P \in Y(k)$. 

In the case that $Y = C_1\times \cdots \times C_d \times A$, by the universal property, 
\[
    \Alb_Y = J_1\times \cdots \times J_d \times A
\]
and 
\[
    \varphi = \iota_{p_1}\times \cdots \times \iota_{p_d}\times \text{id}_A.
\]
is the map $\varphi: Y \to \Alb_Y$ in the above definition.
For more information on the Albanese variety and the map $\alb_Y$ see \cite{bloch2010lectures}.

\section{The Filtration}\label{section: Filtration}

In this section, we give a filtration of the Chow group $\ch_0(C_1\times \ldots \times C_d\times A)$ by defining maps from the Chow groups to the Somekawa K-groups defined in section 2.1. Recall that for each curve $C_i$, we have fixed $k$-points $p_i\in C_i(k)$ and maps $\iota_{p_i} : C_i \to J_i$ such that $\iota_{p_i}(p_i) = 0$. We will identify $J_i(L)$ with $\text{Pic}^0(C_i\times_k L)$ throughout. Under this identification, $\iota_{p_i}$ can be described on $k$-rational points by $\iota_{p_i}(y) = [y]-[p_i]$ for $y \in C_i(k)$. For the rest of this paper we will denote $X := C_1\times \cdots\times C_d\times A$ and $\varphi_X: X \to \Alb_X$ the map to the Albanese variety. We will also denote the identity element of $A(k)$ by $0$ rather than $0_A$ when it is clear from context that it is an element of $A(k)$.  

\begin{rem}
    In the case that any $C_i$ has genus $0$, since we assume that $C_i$ has a $k$-rational point, it is isomorphic to $\mathbb{P}^1$ over $k$. In this case we have that 
    \[
        \ch_0(C_1 \times \cdots \times \mathbb{P}^1 \times \cdots \times C_d\times A) \cong \ch_0(C_1 \times \cdots \times \hat{C}_i \times \cdots \times C_d \times A)
    \]
    (see \cite{fulton2013intersection} Example 1.10.2). Thus, in what follows, we assume that $C_i$ has positive genus for each $i$ and therefore embeds into its Jacobian via $\iota_{p_i}$. In particular, the map $\varphi_X$ is a closed embedding.
\end{rem}

\begin{defn}
    Define the map 
    \[
        \Phi_r: \ch_0(X) \to S_r(k; \underline{J_1}\times \cdots \times \underline{J_d}\times A)
    \]
    by 
    \[
        \Phi_r([z]) = \{\varphi_X(z), \ldots, \varphi_X(z)\}_{k(z)/k}.
    \]
    In the case $r=0$, we take $\Phi_0$ to be the degree map. Here we interpret $\varphi_X(z)$ as a $k(z)$ rational point ($\varphi_X$ is a closed embedding since the genus of each curve is positive, so $k(\varphi_X(z))$ can be identified with $k(z)$). 
\end{defn}

\begin{prop}
    The map $\Phi_r$ is well-defined, i.e. it factors through rational equivalence.
\end{prop}

\begin{proof}
    By Proposition 3.1 of \cite{gazaki2015filtration} for any abelian variety $B$ over $k$, there is a well defined map 
    \[
        \tilde{\Phi}_r : \ch_0(B) \to S_r(k; B)
    \]
    \[
        \tilde{\Phi}_r([b]) = \{b, \ldots, b\}_{k(b)/k}.
    \]
    Taking $B$ to be $J_1\times \ldots \times J_d\times A$ we can realize $\Phi_r$ as the composite 
    \[
        \begin{tikzcd}
            \ch_0(X) \arrow{r}{(\varphi_X)_*} &[1em] \ch_0(\Alb_X) \arrow{r}{\tilde{\Phi}_r} & S_r(k; \Alb_X) \arrow{r}{q} & S_r(k; \underline{J_1}\times\ldots \times \underline{J_d} \times A).
        \end{tikzcd}
    \]
    where $q$ is the natural quotient. Note that $(\varphi_X)_*([z]) = [\varphi_X(z)]$ since $\varphi_X$ is a closed embedding. Thus, the map $\Phi_r$ is well-defined.
\end{proof}

\begin{defn}
    We define a filtration on $\ch_0(X)$ by
    \[
        F^r\ch_0(X) := \bigcap_{j=0}^{r-1}\ker(\Phi_j).
    \]
\end{defn}
It is clear from the definition that $F^0\ch_0(X) = \ch_0(X)$ and $F^1\ch_0(X) = A_0(X)$, the kernel of the degree map. To show that this filtration extends the natural one on $\ch_0(X)$, we verify that $F^2\ch_0(X)$ is the kernel of the Albanese map $\alb_X$.

\begin{prop}
    The kernel of the Albanese map $\alb_X$ is equal to $F^2\ch_0(X)$.
\end{prop}
\begin{proof}
    By the definitions, $S_1(k; \underline{J_1}\times \cdots \times \underline{J_d}\times A) \cong K_1(k; J_1\times \cdots J_d\times A)$. Also by Corollary 3.7 of \cite{gazaki2015filtration}, for any abelian variety $B$, the map 
    \[
        K_1(k; B) \to B(k)
    \]
    defined by $\{b\}_{k'/k} \mapsto \Tr_{k'/k}(b)$ is an isomorphism.
    
    Therefore, we have a commutative diagram
    \[
    \begin{tikzcd}
        \ch_0(X) \arrow{r}{\Phi_0 \oplus \Phi_1} \arrow{d}[swap]{\text{deg}\oplus \alb_{X}} &\Z \oplus K_1(k; J_1\times \cdots \times J_d \times A) \arrow{ld}{\text{id}\oplus \cong} \\
        \Z \oplus (J_1\times \cdots \times J_d \times A)(k).
    \end{tikzcd}
    \]
    Then $F^2\ch_0(X\times A)$ is by definition the kernel of the horizontal map. And the Albanese kernel is the kernel of the vertical map, so these two groups agree.
\end{proof}

\section{Proof of the Main Theorem}

In this section, we prove Theorem \ref{thm: main}. The way that we prove this theorem is by defining maps 
\[
    \Psi_r: S_r(k; \underline{J_1}\times \cdots \times \underline{J_d}\times A) \to F^r\ch_0(X)/F^{r+1}\ch_0(X)
\]
and then computing $\Phi_r\circ \Psi_r$. 

\subsection{Definition of $\Psi_r$ and Proof of Theorem \ref{thm: main}}
Let $k'/k$ be a finite extension. Then we define $\pi_{k'/k}: X_{k'} \to X$ to be the natural projection. We also define the projection maps
\[
pr_{i,{k'}}: (C_1)_{k'}\times_{k'} \cdots \times_{k'} (C_d)_{k'}\times_{k'}A_{k'} \to (C_i)_{k'}
\]
and 
\[
pr_{A, k'}: (C_1)_{k'}\times_{k'} \cdots \times_{k'} (C_d)_{k'}\times_{k'}A_{k'} \to A_{k'}.
\]
Throughout this article, we will also abuse notation by denoting an element of 
\[
    \prod_{i=1}^r J_1(k')\oplus \cdots \oplus J_d(k') \oplus A(k')
\]
by $\{(z^1_1,\ldots, z_d^1, a_1), \ldots, (z_1^r, \ldots, z_d^r, a_r)\}_{k'/k}$ even though multilinearity and the projection formula do not apply. Finally, if $\alpha_i \in \ch_0(A_{k'})$ for $i \in I$ a finite indexing set. Then we denote by
\[
    \bigodot_{i\in I} \alpha_i 
\]
the Pontryagin product of all $\alpha_i$. Using this notation, we make the following definition.

\begin{defn}\label{def:Psi-map}
    Define the map 
    \[
        \Psi'_r: \bigoplus_{k'/k}\prod_{i=1}^r J_1(k')\oplus \cdots \oplus J_d(k') \oplus A(k') \to \ch_0(X)
    \]
    by
    \[
        \{(z_1^1, \ldots, z_d^1, a_1), \ldots, (z_1^r,\ldots, z_d^r, a_r)\}_{k'/k} \mapsto  \sum_{s = 0}^d \sum_{1\le i_1 < \ldots < i_s \le d}\sum_{\varphi_I: \{i_1, \ldots, i_s\} \hookrightarrow \{1, \ldots, r\}} (\pi_{k'/k})_*(W)
    \]
    where 
    \begin{align*}
        &W = \left(pr_{1,k'}^*([p_1]_{k'})\cdots pr_{i_1, k'}^*(z^{\varphi_I(i_1)}_{i_1})\cdots pr_{i_s,k'}^*(z^{\varphi_I(i_s)}_{i_s})\cdots pr_{d,k'}^*([p_d]_{k'})\right) \\
        &\cdot pr_{A,k'}^*\left(\bigodot_{j \in \{1,\ldots r\}\setminus \varphi_I(I) }([a_j]_{k'} - [0]_{k'})\right).
    \end{align*}
    The $z^j_i \in J_i(k')$ are viewed as arbitrary elements of $\text{Pic}^0((C_i)_{k'})$ and $I := \{i_1,\ldots, i_s\}$.
\end{defn}
\begin{rem}
    Here the $i_1,\ldots, i_s$ specify a subset of the curves $C_1,\ldots, C_d$ and the injection $\varphi_I : \{i_1,\ldots, i_s\} \hookrightarrow \{1,\ldots, r\}$ assigns to each curve a unique index from $1,\ldots, r$. The reason we demand that $\varphi_I$ is injective corresponds to the fact that in the group $S_r(k; \underline{J_1}\times \ldots\times \underline{J_d}\times A)$ each direct summand has at most one copy of each Jacobian (see Definition \ref{def: main K-group def}). Therefore, we make this requirement so that $\Psi_r'$ will descend to a group homomorphism from the quotient $S_r(k; \underline{J_1}\times \ldots\times \underline{J_d}\times A)$. 
\end{rem}
\begin{exmp}
    Let us consider the map $\Psi_2'$ in the case when $d = 1$. Denote $z_1^1 := \sum_{y\in C}n_y^1[y]$ and $z_1^2 := \sum_{y\in C} n_y^2[y]$. Then we get that 
    \begin{align}
        \Psi_r'(\{(z_1^1, a_1), (z_1^2, a_2)\}_{k'/k} &= [p_1\times a_1+a_2]_k - [p_1\times a_1] - [p_1\times a_2] + [p_1\times 0]\\
        &+ (\pi_{k'/k})_*\left(\sum_{y\in C}n_y^1\left([y\times a_2]_{k'} - [y\times 0]_{k'}\right)\right) \\
        &+ (\pi_{k'/k})_*\left(\sum_{y\in C}n_y^2\left([y\times a_1]_{k'} - [y\times 0]_{k'}\right)\right).
    \end{align}
    Note that line (4.4) corresponds to a zero cycle coming only from the abelian variety. The line (4.5) is a zero cycle which is impacted by $z_1^1$ and $a_2$ while (4.6) is impacted by $z_1^2$ and $a_1$. 

    This example illustrates what occurs in the general definition of $\Psi_r'$. It interpolates between zero cycles coming purely from the abelian variety and those which come from a mix of some subset of the curves $C_1,\ldots, C_d$ and the abelian variety. 
\end{exmp}

\begin{prop}\label{prop:PsiWellDef}
    The map $\Psi'_r$ is well defined.
\end{prop}

\begin{proof}
    This follows from the fact that the intersection product and Pontryagin product determine well-defined maps on Chow groups, but we provide the full details of this argument here.
    
    Let $\tilde{z}_i^t := z_i^t + \dv f_i^t$ for $1\le i \le d$ and $1 \le t \le r$ where $f_i^t \in K(C_i\times_k k')^\times$. Let $a_1, \ldots, a_r \in A(k')$. Then we need to show that 
    \begin{align*}
        &\Psi_r'\{(\tilde{z}_1^1, \ldots, \tilde{z}_d^1, a_1), \ldots, (\tilde{z}_1^r,\ldots, \tilde{z}_d^r, a_r)\}_{k'/k} \\
        &= \Psi_r'\{(z_1^1, \ldots, z_d^1, a_1), \ldots, (z_1^r,\ldots, z_d^r, a_r)\}_{k'/k}. 
    \end{align*}
    Let us assume that $k'=k$ as the argument is the same in the general case. We have 
    \begin{align*}
        &\Psi_r'\left(\{(\tilde{z}_1^1, \ldots, \tilde{z}_d^1, a_1), \ldots, (\tilde{z}_1^r,\ldots, \tilde{z}_d^r, a_r)\}_{k/k}\right) \\
        &= \sum_{s = 0}^d \sum_{1\le i_1 < \ldots < i_s \le d}\sum_{\varphi_I: \{i_1, \ldots, i_s\} \hookrightarrow \{1, \ldots, r\}}  \\
        & \left(pr_{1,k}^*([p_1])\cdots pr_{i_1,k}^*\left(z^{\varphi_I(i_1)}_{i_1} + \dv f_{i_1}^{\varphi_I(i_1)}\right)\cdots pr_{i_s,k}^*\left(z^{\varphi_I(i_s)}_{i_s} + \dv f_{i_s}^{\varphi_I(i_s)}\right)\cdots pr_{d,k}^*([p_d])\right) \\
        &\cdot pr_{A,k}^*\left(\bigodot_{j \in \{1,\ldots r\}\setminus \varphi_I(I) }([a_j] - [0])\right).
    \end{align*}
    But after expanding using multilinearity of the intersection product and the fact that $pr_{j,k}^*(\dv f_j^i) = 0$ for any $j,i$, we have that the above is equal to 
    \begin{align*}
    &\sum_{s = 0}^d \sum_{1\le i_1 < \ldots < i_s \le d}\sum_{\varphi_I: \{i_1, \ldots, i_s\} \hookrightarrow \{1, \ldots, r\}} \\
    & \left(pr_{1,k}^*([p_1])\cdots pr_{i_1,k}^*(z^{\varphi_I(i_1)}_{i_1})\cdots pr_{i_s,k}^*(z^{\varphi_I(i_s)}_{i_s})\cdots pr_{d,k}^*([p_d])\right) \\
    &\cdot  pr_{A,k'}^*\left(\bigodot_{j \in \{1,\ldots r\}\setminus \varphi_I(I) }([a_j] - [0])\right)\\
    &= \Psi_r'\{(z_1^1, \ldots, z_d^1, a_1), \ldots, (z_1^r,\ldots, z_d^r, a_r)\}_{k/k}.
    \end{align*}
    So the map is well defined.
\end{proof}

The most involved aspect in proving Theorem \ref{thm: main} is verifying the following two key lemmas. We give the proofs in section \ref{section: proofs} and use the results in this section to show the main theorem.

\begin{lem}\label{lem: key-lem-Psi}
    The map $\Psi_r'$ induces a group homomorphism 
    \[
        \Psi_r: S_r(k; \underline{J_1}\times \cdots \times \underline{J_d}\times A) \to F^r\ch_0(X)/F^{r+1}\ch_0(X).
    \]
\end{lem}
The content of this lemma is that the image of the map $\Psi_r'$ lands in $F^r\ch_0(X)$ and when we take the composite 
\[
    \bigoplus_{k'/k}\prod_{i=1}^r J_1(k')\oplus \ldots \oplus J_d(k') \oplus A(k') \xrightarrow{\Psi_r'} F^r\ch_0(X) \to F^{r}\ch_0(X)/F^{r+1}\ch_0(X),
\]
which we denote $\Psi_r$, this map induces a well-defined group homomorphism on the quotient group $S_r(k; \underline{J_1}\times \cdots \times \underline{J_d}\times A)$.

\begin{lem}\label{lem: key-lem-comp}
    The composite 
    \[
        S_r(k; \underline{J_1}\times \cdots \times \underline{J_d}\times A) \xrightarrow{\Psi_r} F^r\ch_0(X)/F^{r+1}\ch_0(X) \xrightarrow{\Phi_r} S_r(k; \underline{J_1}\times \cdots \times \underline{J_d}\times A)
    \]
    is multiplication by $r!$.
\end{lem}

Assuming these lemmas, we are able to prove the main theorem.
\begingroup
\def\thethm{\ref{thm: main}}
\begin{thm}
    The map $\Phi_r$ induces an isomorphism
    \[
        \mathbb{Z}[1/r!]\otimes \frac{F^r(X\times A)}{F^{r+1}(X\times A)} \xrightarrow{\simeq} \mathbb{Z}[1/r!]\otimes S_r(k; \underline{J_1}\times \cdots \times \underline{J_d}\times A).
    \]
\end{thm}
\addtocounter{thm}{-1}
\endgroup

\begin{proof}
    We have a short exact sequence 
    \[
        \begin{tikzcd}
            0 \arrow{r} & \frac{F^r(X)}{F^{r+1}(X)} \arrow{r}{\Phi_r} & S_r(k; \underline{J_1}\times\cdots \times \underline{J_d}\times A) \arrow{r} & \frac{S_r(k; \underline{J_1}\times \cdots \times \underline{J_d}\times A)}{\text{Im}(\Phi_r)} \arrow{r} & 0.
        \end{tikzcd}
    \]
    But by Lemma \ref{lem: key-lem-comp}, we have that the last group is $r!$ torsion, so it will become 0 after tensoring with $\Z[1/r!]$. Thus we have
    \[
        \begin{tikzcd}
            \frac{F^r(X)}{F^{r+1}(X)}\otimes \mathbb{Z}[1/r!] \arrow{r}{\Phi_r} & S_r(k; \underline{J_1}\times \cdots \times\underline{J_d}\times A) \arrow{r} \otimes \mathbb{Z}[1/r!] &  0.
        \end{tikzcd}
    \]
    But $\Z[1/r!]$ is torsion free and hence it is a flat $\Z$ module. Thus, the first map is injective as well, which proves the isomorphism. 
\end{proof}

\subsection{Proof of Key Lemmas}\label{section: proofs}
The purpose of this section is to prove Lemmas \ref{lem: key-lem-Psi} and \ref{lem: key-lem-comp}. We begin with a few preliminary results which will establish important properties of $\Phi_r$ and $\Psi'_r$. These properties will be used throughout the rest of this article. The first fact is that the map $\Phi_r$ behaves well with the trace map on Somekawa $K$-groups.
\begin{prop}\label{prop:Phi-Tr-commute}
    Let $\alpha \in \ch_0(X_{k'})$. Then 
    \[
        \Phi_r((\pi_{k'/k})_*(\alpha)) = \Tr_{k'/k}(\Phi_r^{k'}(\alpha))
    \]
\end{prop}
\begin{proof}
    Because all of these maps are $\Z$-linear, it suffices to take $\alpha = [x]_{k'}$ for $x \in X_{k'}$ a closed point. Let $b := \pi_{k'/k}(x) \in X$, $i: X \to \Alb_X$ be the closed embedding, and $i_{k'}: X_{k'} \to \Alb_{X_{k'}}$ be its base change. Then we get 
    \begin{align*}
        \Phi_r((\pi_{k'/k})_*([x]_{k'})) &= \Phi_r([k(x):k(b)][b]) \\
        &= [k(x): k(b)]\{i(b), \ldots, i(b)\}_{k(b)/k} \\
        &= \{[k(x): k(b)]i(b), \ldots, i(b)\}_{k(b)/k} \\
        &= \{\Tr_{k(x)/k(b)}\res_{k(x)/k(mb)}i(b), \ldots, i(b)\}_{k(b)/k} \\
        &= \{\res_{k(x)/k(b)}i(b), \ldots, \res_{k(x)/k(b)}i(b)\}_{k(x)/k} \qquad\text{by Projection Formula}
    \end{align*}
    and 
    \begin{align*}
        \Tr_{k'/k}\Phi^{k'}_{r}([x]_{k'}) &= \Tr_{k'/k}\{i_{k'}(x), \ldots, i_{k'}(x)\}_{k(x)/k'} \\
        &= \{\pi_{k'/k}\circ i_{k'}(x), \ldots, \pi_{k'/k}\circ i_{k'}(x)\}_{k(x)/k} \\
        &= \{i\circ \pi_{k'/k}\circ x, \ldots, i\circ \pi_{k'/k}\circ x\}_{k(x)/k} \\
        &= \{\res_{k(x)/k(b)}(i(b)), \ldots, \res_{k(x)/k(b)}(i(b))\}_{k(x)/k}.
    \end{align*}
\end{proof}

The second preliminary result says that $\Psi_r'$ satisfies a partial version of multilinearity (we will verify $\Psi_r$ is fully multilinear later). 

\begin{prop}\label{prop: partial multilinearity}
    Fix some $1\le t \le r$. For $1 \le j \le r$, $1 \le i \le d$ $a_j \in A(k')$, $z_i^j \in J_i(k')$ and $\tilde{z}_i^t \in J_i(k')$. Then 
    \begin{align*}
        &\Psi_r'\left(\{(z_1^1, \ldots, z_d^1, a_1), \ldots,(z_1^t+\tilde{z}_1^t,\ldots, z_d^t+\tilde{z}_d^t, a_t),\ldots, (z_1^r,\ldots, z_d^r, a_r)\}_{k'/k}\right) \\
        &= \Psi_r'\left(\{(z_1^1, \ldots, z_d^1, a_1), \ldots,(z_1^t,\ldots, z_d^t, 0),\ldots, (z_1^r,\ldots, z_d^r, a_r)\}_{k'/k}\right) \\
        &+ \Psi_r'\left(\{(z_1^1, \ldots, z_d^1, a_1), \ldots,(\tilde{z}_1^t,\ldots, \tilde{z}_d^t, a_t),\ldots, (z_1^r,\ldots, z_d^r, a_r)\}_{k'/k}\right).
    \end{align*}
    So $\Psi_r'$ satisfies multilinearity except for in the abelian variety components. 
\end{prop}

\begin{proof}
    Since $\Psi_r'$ is invariant under the action of the symmetric group, it suffices to show the proposition for $t = 1$. We show this in two steps: For step one we show that 
    \begin{align*}
        &\Psi_r'\left(\{(z_1^1, \ldots, z_d^1, a_1),\ldots, (z_1^r,\ldots, z_d^r, a_r)\}_{k'/k}\right) \\
        &= \Psi_r'\left(\{(z_1^1, \ldots, z_d^1, 0),(z_1^2,\ldots, z_d^2, a_2),\ldots, (z_1^r,\ldots, z_d^r, a_r)\}_{k'/k}\right)\\
        &+ \Psi_r'\left(\{(0, \ldots, 0, a_1),(z_1^2,\ldots, z_d^2, a_2),\ldots, (z_1^r,\ldots, z_d^r, a_r)\}_{k'/k}\right)
    \end{align*}
    and for step two we show that 
    \begin{align*}
        &\Psi_r'\left(\{(z_1^1+ \tilde{z}_1^1, \ldots, z_d^1+\tilde{z}_d^1, 0),(z_1^2,\ldots, z_d^2, a_2),\ldots, (z_1^r,\ldots, z_d^r, a_r)\}_{k'/k}\right) \\
        &= \Psi_r'\left(\{(z_1^1, \ldots, z_d^1, 0),(z_1^2,\ldots, z_d^2, a_2),\ldots, (z_1^r,\ldots, z_d^r, a_r)\}_{k'/k}\right)\\
        &+ \Psi_r'\left(\{(\tilde{z}_z^1, \ldots, \tilde{z}_d^1, 0),(z_1^2,\ldots, z_d^2, a_2),\ldots, (z_1^r,\ldots, z_d^r, a_r)\}_{k'/k}\right).
    \end{align*}
    To verify step one, note that by definition
    \begin{align*}
        &\Psi_r'\left(\{(z_1^1, \ldots, z_d^1, 0),(z_1^2,\ldots, z_d^2, a_2),\ldots, (z_1^r,\ldots, z_d^r, a_r)\}_{k'/k}\right) \\
        &=\sum_{s = 0}^d \sum_{1\le i_1 < \ldots < i_s \le d}\sum_{\varphi_I: \{i_1, \ldots, i_s\} \hookrightarrow \{1, \ldots, r\}} (\pi_{L/k})_*(W)
    \end{align*}
    where 
    \begin{align*}
        &W = \left(pr_{1,k'}^*([p_1]_{k'})\cdots pr_{i_1, k'}^*(z^{\varphi_I(i_1)}_{i_1})\cdots pr_{i_s,k'}^*(z^{\varphi_I(i_s)}_{i_s})\cdots pr_{d,k'}^*([p_d]_{k'})\right) \\
        &\cdot pr_{A,k'}^*\left(\bigodot_{j \in \{1,\ldots r\}\setminus \varphi_I(I) }([a_j]_{k'} - [0]_{k'})\right).
    \end{align*}
    However, since $a_1 = 0$, if $1$ is not in the image of $\varphi_I$, then $([a_1]_{k'} - [0]_{k'}) = 0$, so we may take the sum $\sum_{\varphi_I: \{i_1, \ldots, i_s\} \hookrightarrow \{1, \ldots, r\}}$ to be over all injections such that $1$ is in the image of $\varphi_I$. Also,
    \begin{align*}
        &\Psi_r'\left(\{(0, \ldots, 0, a_1),(z_1^2,\ldots, z_d^2, a_2),\ldots, (z_1^r,\ldots, z_d^r, a_r)\}_{k'/k}\right) \\
        &\sum_{s = 0}^d \sum_{1\le i_1 < \ldots < i_s \le d}\sum_{\varphi_I: \{i_1, \ldots, i_s\} \hookrightarrow \{2, \ldots, r\}} (\pi_{L/k})_*(W).
    \end{align*}
    Note that the sum is now taken over all injections $\varphi_I: \{i_1,\ldots, i_s\} \hookrightarrow \{2,\ldots, r\}$ since if $1$ is in the image of $\varphi_I$ then a $z_t^1 = 0$ term appears in the intersection product. Therefore, 
    \[
        \Psi_r'\left(\{(z_1^1, \ldots, z_d^1, 0),(z_1^2,\ldots, z_d^2, a_2),\ldots, (z_1^r,\ldots, z_d^r, a_r)\}_{k'/k}\right)
    \]
    is the sum over all injections $\varphi_I$ such that $1$ is in its image and 
    \[
        \Psi_r'\left(\{(0,\ldots,0, a_1),(z_1^2,\ldots, z_d^2, a_2),\ldots, (z_1^r,\ldots, z_d^r, a_r)\}_{k'/k}\right)
    \]
    is the sum over all injections $\varphi_I$ such that $1$ is not in its image. Hence their sum is the sum over all injections $\varphi_I$ which is by definition
    \[
        \Psi_r'\left(\{(z_1^1, \ldots, z_d^1, a_1),\ldots, (z_1^r,\ldots, z_d^r, a_r)\}_{k'/k}\right).
    \]
    To verify step two note that 
    \[
        \Psi_r'\left(\{(z_1^1+ \tilde{z}_1^1, \ldots, z_d^1+\tilde{z}_d^1, 0),(z_1^2,\ldots, z_d^2, a_2),\ldots, (z_1^r,\ldots, z_d^r, a_r)\}_{k'/k}\right)
    \]
    is the sum over all injections $\varphi_I$ where $1$ is in the image by the argument above. Since the intersection product is linear in each factor, we get that 
    \begin{align*}
        &\Psi_r'\left(\{(z_1^1+ \tilde{z}_1^1, \ldots, z_d^1+\tilde{z}_d^1, 0),(z_1^2,\ldots, z_d^2, a_2),\ldots, (z_1^r,\ldots, z_d^r, a_r)\}_{k'/k}\right) \\
        &= \Psi_r'\left(\{(z_1^1, \ldots, z_d^1, 0),(z_1^2,\ldots, z_d^2, a_2),\ldots, (z_1^r,\ldots, z_d^r, a_r)\}_{k'/k}\right)\\
        &+ \Psi_r'\left(\{(\tilde{z}_z^1, \ldots, \tilde{z}_d^1, 0),(z_1^2,\ldots, z_d^2, a_2),\ldots, (z_1^r,\ldots, z_d^r, a_r)\}_{k'/k}\right).
    \end{align*}
\end{proof}

The next result that we need is that the map $\Psi_r'$ respects a particular case of the projection formula (we will verify that $\Psi_r$ respects the full version of the projection formula later). 

\begin{prop}\label{prop:partial proj form}
    Let $L/E/k$ be a tower of finite extensions and fix $1\le t \le r$. If $z_j^t \in Pic^0((C_i)_L) \cong J_j(L)$ for $1\le j\le d$ and for $i\ne t$, $z_j^i\in Pic^0((C_j)_E) \cong J_j(E)$  for $1\le j \le d$ and $a_i \in A(E)$, then
    \begin{align*}
        &\Psi_r'\{\res_{L/E}(z_1^1, \ldots, z_d^1, a_1), \ldots,(z_1^t,\ldots, z_d^t, 0) ,\ldots, \res_{L/E}(z_1^r,\ldots, z_d^r, a_r)\}_{L/k} \\
        &= \Psi_r'\{(z_1^1, \ldots, z_d^1, a_1), \ldots,\Tr_{L/E}(z_1^t,\ldots, z_d^t, 0) ,\ldots (z_1^r,\ldots, z_d^r, a_r)\}_{E/k}.
    \end{align*}
\end{prop}

\begin{proof}
    This results follows from the projection formula for cycles which we recall here. Let $f: Y \to Z$ be a flat, proper morphism. Let $\alpha \in \ch^*(Y)$ and $\beta \in \ch^*(Z)$. Then 
    \[
        f_*(\alpha\cdot f^*(\beta)) = f_*(\alpha)\cdot \beta
    \]
    where $f_*$ is the proper pushforward and $f^*$ is the flat pullback.

    By Proposition \ref{prop: partial multilinearity} and the fact that $\Psi_r'$ is symmetric in the order of its components, it suffices to show that
    \begin{align*}
        &\Psi_r'\{(0,\ldots,0, z_{i_1}, 0, \ldots, 0), \ldots,\res_{L/E}(0,\ldots, 0, z_{i_s}, 0,\ldots 0), \\
        &\res_{L/E}(0,\ldots, a_1),\ldots, \res_{L/E}(0,\ldots, 0, a_{r-s})\}_{L/k} \\
        &= \Psi_r'\{\Tr_{L/E}(0,\ldots,0, z_{i_1}, 0, \ldots, 0), \ldots,(0,\ldots, 0, z_{i_s}, 0,\ldots 0),(0,\ldots, a_1),\ldots, (0,\ldots, 0, a_{r-s})\}_{E/k}.
    \end{align*}
    For some $1\le i_1 < \ldots < i_s \le d$. Note that
    \[
        \Tr_{L/E}(z_1^1,\ldots, z_d^1, 0) = ((\pi_{L/E})_*(z_1^1),\ldots, (\pi_{L/E})_*(z_d^1), 0)
    \]
    and 
    \[
        \res_{L/E}(z_1^1,\ldots, z_d^1, a) = ((\pi_{L/E})^*(z_1^1),\ldots, (\pi_{L/E})^*(z_d^1), \res_{L/E}(a)).
    \]
    Thus by the definition of $\Psi_r'$, we get 
    \begin{align*}
        &\Psi_r'\{(0,\ldots,0, z_{i_1}, 0, \ldots, 0), \ldots,\res_{L/E}(0,\ldots, 0, z_{i_s}, 0,\ldots 0), \\
        &\res_{L/E}(0,\ldots, a_1),\ldots, \res_{L/E}(0,\ldots, 0, a_{r-s})\}_{L/k} =  (\pi_{L/k})_*(W)
    \end{align*}
    where 
    \begin{align*}
        &W = \left(pr_{1,L}^*([p_1]_{L})\cdots pr_{i_1, L}^*(z_{i_1})\cdots pr_{i_2, L}^*(\pi_{L/E}^*z_{i_1})\cdots pr_{i_s,L}^*(\pi_{L/E}^*z_{i_s})\cdots pr_{d,L}^*([p_d]_{L})\right) \\
        &\cdot pr_{A,L}^*\left(\bigodot_{j \in \{1,\ldots r-s\}}([a_j]_{L} - [0]_{L})\right)
    \end{align*}
    Note that for all $t, pr_{t,L}^*\circ \pi_{L/E}^* = \pi_{L/E}^*\circ pr_{t, E}^*$ and similarly, $pr_{A,L}^*\circ \pi_{L/E}^* = \pi_{L/E}^*\circ pr_{A, E}^*$. Therefore, we can write 
    \begin{align*}
        &W = pr_{i_1,L}^*(z_{i_1})\cdot\pi_{L/E}^*\left[\vphantom{\bigodot_{j=1}}\left(pr_{1,E}^*([p_1]_{E})\cdots \widehat{pr_{i_1, E}^*(z_{i_1})}\cdots pr_{i_2, E}^*(z_{i_1})\cdots pr_{i_s,E}^*(z_{i_s})\cdots pr_{d,E}^*([p_d]_{E})\right)\right. \\
        &\left.\cdot pr_{A,E}^*\left(\bigodot_{j \in \{1,\ldots r-s\}}([a_j]_{E} - [0]_{E})\right)\right].
    \end{align*}
    Since $(\pi_{L/E})_*\circ pr_{1,L}^* = pr_{1,E}^* \circ (\pi_{L/E})_*$, by the projection formula we have 
    \begin{align*}
        &\Psi_r'\{(0,\ldots,0, z_{i_1}, 0, \ldots, 0), \ldots,\res_{L/E}(0,\ldots, 0, z_{i_s}, 0,\ldots 0), \\
        &\res_{L/E}(0,\ldots, a_1),\ldots, \res_{L/E}(0,\ldots, 0, a_{r-s})\}_{L/k} =  (\pi_{E/k})_*(Q)
    \end{align*}
    where 
    \begin{align*}
        &Q = pr_{i_1,E}^*((\pi_{L/E})_*z_{i_1})\cdot\vphantom{\bigodot_{j=1}}\left(pr_{1,E}^*([p_1]_{E})\cdots \widehat{pr_{i_1, E}^*(z_{i_1})}\cdots pr_{i_2, E}^*(z_{i_1})\cdots pr_{i_s,E}^*(z_{i_s})\cdots pr_{d,E}^*([p_d]_{E})\right)\\
        &\cdot pr_{A,E}^*\left(\bigodot_{j \in \{1,\ldots r-s\}}([a_j]_{E} - [0]_{E})\right).
    \end{align*}
    And this is equal to 
    \[
        \Psi_r'\{\Tr_{L/E}(0,\ldots,0, z_{i_1}, 0, \ldots, 0), \ldots,(0,\ldots, 0, z_{i_s}, 0,\ldots 0),(0,\ldots, a_1),\ldots, (0,\ldots, 0, a_{r-s})\}_{E/k}
    \]
    which completes the proof.
\end{proof}

\begin{rem}\label{rem: Reduction of proofs}
    The way that Propositions \ref{prop: partial multilinearity} and \ref{prop:partial proj form} will be used is that the properties we wish to verify below for the map $\Psi_r'$ can now be checked on symbols of the form
    \begin{align*}
        &\{(0, \ldots, 0, [y_{i_1}]-[p_{i_1}]_{k'}, 0 \ldots, 0), \ldots, \left(0, \ldots, 0,[y_{i_s}]-[p_{i_s}]_{k'}, 0 \ldots, 0\right), \\
        & (0,\ldots,0, a_1), \ldots, (0,\ldots, 0, a_{r-s})\}_{k'/k}
    \end{align*}
    where $1 \le i_1 < \ldots < i_s \le d$, and the $y_t$ are $k'$ rational points of $(C_t)_{k'}$. The way that this reduction is made is that starting with a general symbol, we can use Proposition \ref{prop: partial multilinearity} and the fact that $\Psi_r'$ is invariant under the action of the symmetric group to write its image under $\Psi_r'$ as a sum of cycles of the form 
    \begin{align*}
        &\Psi_r'(\{(0, \ldots, 0, [y_{i_1}]-[k(y_{i_1}):k'][p_{i_1}]_{k'}, 0 \ldots, 0), \ldots, \left(0, \ldots, 0,[y_{i_s}]-[k(y_{i_s}):k'][p_{i_s}]_{k'}, 0 \ldots, 0\right), \\
        & (0,\ldots,0, a_1), \ldots, (0,\ldots, 0, a_{r-s})\}_{k'/k})
    \end{align*}
    where each $y_t$ is a closed point of $(C_t)_{k'}$. But then we note that 
    \[
        (0, \ldots, 0, [y_{i_1}]-[k(y_{i_1}):k'][p_{i_1}]_{k'}, 0 \ldots, 0) = \Tr_{k(y_{i_1})/k}(0, \ldots, 0, [y_{i_1}]_{k(y_{i_1})}-[p_{i_1}]_{k(y_{i_1})}, 0 \ldots, 0).
    \]
    We can then apply Proposition \ref{prop:partial proj form}, replace $k'$ with $k(y_{i_1})$ and repeat the process with the remaining indices $i_2,\ldots, i_s$ to note that $\Psi_r'$ applied to the original symbol is a sum of cycles of the form 
    \begin{align*}
        &\Psi_r'\{(0, \ldots, 0, [y_{i_1}]-[p_{i_1}]_{k'}, 0 \ldots, 0), \ldots, \left(0, \ldots, 0,[y_{i_s}]-[p_{i_s}]_{k'}, 0 \ldots, 0\right), \\
        & (0,\ldots,0, a_1), \ldots, (0,\ldots, 0, a_{r-s})\}_{k'/k}
    \end{align*}
    where each $y_t$ is a $k'$ rational point of $(C_t)_{k'}$. Finally note that if $i_p = i_q$ for $p \ne q$ i.e. $y_{i_p}$ and $y_{i_q}$ are points on the same curve, then the definition of $\Psi_r'$ implies that the image of that symbol is $0$. So we may take $1 \le i_1 < \ldots < i_s \le d$. 
\end{rem}

The last preliminary result that we will need is the following alternative definition for $\Psi_r'$.

\begin{prop}\label{prop: psi alt def}
    An equivalent definition of $\tilde{\Psi'}_r$ is given by
    \begin{align*}
        &\left\{\left(\sum_{y_1}n_{y_1}^1[y_1], \ldots, \sum_{y_d}n_{y_d}^1[y_d],a_1\right), \ldots, \left(\sum_{y_1}n_{y_1}^r[y_1], \ldots, \sum_{y_d}n_{y_d}^r[y_d],a_r\right)\right\}_{k'/k} \mapsto \\
        &\sum_{s = 0}^d \sum_{1\le i_1 < \ldots < i_s\le d}\sum_{\varphi_I : I \hookrightarrow \{1,\ldots, r\}} \sum_{y_{i_1}} \ldots \sum_{y_{i_s}} n_{y_{i_1}}^{\varphi_I(i_1)} \cdots n_{y_{i_s}}^{\varphi_I(i_s)} \sum_{j=0}^{r-s} (-1)^{r-s-j} \\
        &(\pi_{k'/k})_*\left(\sum_{1 \le \nu_1 < \ldots \widehat{\varphi_I(I)} \ldots < \nu_j \le r} [p_1\times_{k'} \ldots \times_{k'} y_{i_1}\times_{k'} \ldots\times_{k'} y_{i_s}\times_{k'} \ldots\times_{k'} p_d\times_{k'} (a_{\nu_1}+ \ldots + a_{\nu_j})]_{k'} \right).
    \end{align*}
     Note that in general, $p_1\times_{k'} \ldots, y_{i_1}\times_{k'} \ldots\times_{k'} y_{i_k}\times_{k'} \ldots\times_{k'} p_d\times_{k'} (a_{\nu_1}+ \ldots + a_{\nu_j})$ is not a closed point of $(C_1)_{k'}\times_{k'} \ldots \times_{k'}(C_d)_{k'}$, but we take the cycle class which the subscheme defines (see \cite{fulton2013intersection} Chapter 1 section 5). Also, the notation $1 \le \nu_1 < \ldots \widehat{\varphi_I(I)} \ldots < \nu_j \le r$ means that $\{\nu_1,\ldots, \nu_j\} \subset \{1,\ldots, r\} \setminus \varphi_I(I)$. 
\end{prop}

\begin{proof}
    Let $z_i^j = \sum_{y_i}n^j_{y^i}[y_i]$. Then by the multilinearity of the intersection product, we have 
    \begin{align*}
        &\left(pr_{1,k'}^*([p_1]_{k'})\cdots pr_{i_1,k'}^*(z^{\varphi_I(i_1)}_{i_1})\cdots pr_{i_s, k'}^*(z^{\varphi_I(i_s)}_{i_s})\cdots pr_{d, k'}^*([p_d]_{k'})\right) \\
        &\cdot pr_{A,k'}^*\left(\bigodot_{j \in \{1,\ldots, r\}\setminus \varphi_I(I)} ([a_j]_{k'} - [0]_{k'})\right) \\
        &= \sum_{y_{i_1}} \ldots \sum_{y_{i_s}} n_{y_{i_1}}^{\varphi_I(i_1)} \cdots n_{y_{i_s}}^{\varphi_I(i_s)} \\
        &(pr_{1, k'}^*([p_1]_{k'})\cdots pr_{i_1, k'}^*([y_{i_1}])\cdots pr_{i_s, k'}^*([y_{i_s}])\cdots pr_{d, k'}^*([p_d]_{k'})) \\
        &\cdot pr_{A, k'}^*\left(\bigodot_{j \in \{1,\ldots, r\}\setminus \varphi_I(I)} ([a_j]_{k'} - [0]_{k'})\right).
    \end{align*}
    Note that 
    \begin{align*}
        & \bigodot_{j \in \{1,\ldots, r\}\setminus \varphi_I(I)} ([a_j]_{k'} - [0]_{k'}) \\
        &=([a_1]_{k'}-[0]_{k'})\odot \cdots \odot \widehat{([a_{\varphi_I(i_1)}]_{k'}-[0]_{k'})}\odot \cdots  \odot \widehat{([a_{\varphi_I(i_s)}]_{k'} - [0]_{k'})} \odot \cdots \odot ([a_r]_{k'}-[0]_{k'}) \\
        &= \sum_{j=0}^{r-s} (-1)^{r-s-j}\sum_{1 \le \nu_1 < \ldots \widehat{\varphi_I(I)} \ldots < \nu_j \le r}[a_{\nu_1}]_{k'}\odot \cdots \odot [a_{\nu_j}]_{k'}\odot ([0]_{k'})^{r-s-j}\qquad\text{By multilinearity of $\odot$} \\
        &= \sum_{j=0}^{r-s} (-1)^{r-s-j} \sum_{1 \le \nu_1 < \ldots \widehat{\varphi_I(I)} \ldots < \nu_j \le r} [a_{\nu_1}+ \ldots + a_{\nu_j}]_{k'} \qquad \text{By definition of $\odot$.}
    \end{align*}

    Then it remains to observe that if $Z, Y$ are smooth projective varieties and $\alpha \in \ch(Z)$, $\beta \in \ch(Y)$, then $[\alpha\times \beta] = pr_Z^*(\alpha)\cdot pr_Y^*(\beta) \in \ch(Z\times Y)$.
\end{proof}

Now let us turn to the proofs of the key lemmas. We have a well-defined map
\[
    \Psi_r': \bigoplus_{k'/k}\prod_{i=1}^r J_1(k')\oplus \cdots \oplus J_d(k') \oplus A(k') \to \ch_0(X).
\]
Our current goal is to show that this induces a well-defined map
\[
    S_r(k; \underline{J_1}\times \cdots \times \underline{J_d}\times A) \to F^r\ch_0(X)/F^{r+1}\ch_0(X).
\]
The first step is to show that the image of $\Psi_r'$ lands in $F^r\ch_0(X)$.
\begin{prop}\label{prop:Psi-in-Fr}
    The image of $\Psi'_r$ is contained in $F^r\ch_0(X)$. That is, $\Phi_t \circ \tilde{\Psi'}_r = 0$ for $t< r$.
\end{prop}
We first prove this proposition in the case when $d=1$ to ease the notation. This will contain the main ideas of the proof in the general case with additional complexity in one aspect which we will indicate.

\begin{proof}[Proof in the case of one curve]
    We show that $\Phi_{r-1}\circ \Psi_r' = 0$ as the same argument works for any $t < r$. By Propositions \ref{prop: partial multilinearity} and \ref{prop:partial proj form} (see Remark \ref{rem: Reduction of proofs}) it suffices to show that
    \[
        \Phi_{r-1}\circ \Psi_r'\left(\{(0,a_1), \ldots, (0,a_r)\}_{k'/k}\right) = 0
    \]
    and
    \[
        \Phi_{r-1}\circ \Psi_r'\left(\left\{\left(\sum_{y\in C} n_y [y],0\right),(0,a_2) \ldots, (0,a_r)\right\}_{k'/k}\right) = 0
    \]
    where $y$ is a $k'$-rational point whenever $n_y\ne 0$. 

    Let us start with step 1 which is showing that 
    \[
        \Phi_{r-1}\circ \Psi_r'\left(\{(0,a_1), \ldots, (0,a_r)\}_{k'/k}\right) = 0
    \]
    \begin{rem}\label{rem: Part 1 of landing in F^r}
        This part of the argument is combinatorial and relies on keeping track of the symbols which appear in the image of $\Phi_{r-1}$. A similar counting argument we apply here works in the analogous part of the general case. The point is if we have symbols of length $<r$ and $\ge r$ many variables in the $a_i$, then sums of the above form will become $0$. The difference is that in the general case we will have multiple components of the symbols which have elements coming from curves so the counting argument will require more complicated notation.
    \end{rem}
    We have by definition that 
    \[
        \Psi_r'\left(\{(0,a_1), \ldots, (0,a_r)\}_{k'/k}\right) = \sum_{j=0}^{r}(-1)^{r-j}(\pi_{k'/k})_*\left(\sum_{1\le \nu_1< \ldots < \nu_j \le r}[p_1\times_{k'} (a_{\nu_1}+\ldots + a_{\nu_j})]_{k'}\right).
    \]
    By Proposition \ref{prop:Phi-Tr-commute}, we have 
    \begin{align*}
        &\Phi_{r-1}\left( \sum_{j=0}^{r}(-1)^{r-j}(\pi_{k'/k})_*\left(\sum_{1\le \nu_1< \ldots < \nu_j \le r}[p_1\times_{k'} (a_{\nu_1}+\ldots + a_{\nu_j})]_{k'}\right)\right) \\
        &= \sum_{j=0}^{r}(-1)^{r-j}\left(\sum_{1\le \nu_1< \ldots < \nu_j \le r}\{\underbrace{(0, a_{\nu_1}+\ldots + a_{\nu_j}), \ldots, (0, a_{\nu_1}+\ldots + a_{\nu_j})}_{r-1}\}_{k'/k}\right).
    \end{align*}
    Using multilinearity in $S_r(k; \underline{J_1}\times A)$, we can rewrite the above as
    \[
        \sum_{j=0}^{r}(-1)^{r-j}\left(\sum_{1\le \nu_1< \ldots < \nu_j \le r}\sum_{\overrightarrow{\ell} \in \{\nu_1,\ldots, \nu_j\}^{r-1}}\{(0,a_{\ell_1}), \ldots, (0,a_{\ell_{r-1}})\}_{k'/k}\right).
    \]
    For each $\overrightarrow{h} \in \{1,\ldots, r\}^{r-1}$ and $0 \le j \le r$, define 
    \[
        M_{\overrightarrow{h}, j} := |\{X \subset \{1,\ldots, r\} : |X| = j, \overrightarrow{h} \in X^{r-1}\}|.
    \] 
    That is, $M_{\overrightarrow{h}, j}$ is the number of subsets of $\{1,\ldots, r\}$ of size $j$ that contain every component of $\overrightarrow{h}.$ Note that for a fixed $\overrightarrow{h} \in \{1,\ldots, r\}$, the number of times the symbol
    \[
        \{(0, a_{h_1}), \ldots, (0, a_{h_{r-1}})\}_{k'/k}
    \]
    appears in the sum 
    \[
        \sum_{1\le \nu_1< \ldots < \nu_j \le r}\sum_{\overrightarrow{\ell} \in \{\nu_1,\ldots, \nu_j\}^{r-1}}\{(0,a_{\ell_1}), \ldots, (0,a_{\ell_{r-1}})\}_{k'/k}
    \]
    is equal to $M_{\overrightarrow{h}, j}$.
    
    Putting the above together, we get that
    \begin{align*}
        &\Phi_{r-1}\left( \sum_{j=0}^{r}(-1)^{r-j}(\pi_{k'/k})_*\left(\sum_{1\le \nu_1< \ldots < \nu_j \le r}[p_1\times_{k'} (a_{\nu_1}+\ldots + a_{\nu_j})]_{k'}\right)\right) \\
        &= \sum_{\overrightarrow{h} \in \{1, \ldots, r\}^{r-1}}\left(\sum_{j=0}^r(-1)^{r-j} M_{\overrightarrow{h}, j} \{(0,a_{h_1}), \ldots, (0,a_{h_{r-1}})\}_{k'/k}\right).
    \end{align*}
    Thus, it suffices to show that for fixed $\overrightarrow{h} \in \{1,\ldots, r\}^{r-1}$, $\sum_{j=0}^r (-1)^{r-j} M_{\overrightarrow{h}, j} = 0$. A counting argument shows that
    \[
        M_{\overrightarrow{h}, j} = {r-|\{h_1,\ldots, h_{r-1}\}| \choose j-|\{h_1,\ldots, h_{r-1}\}|},
    \]
    so we get 
    \[
        \sum_{j=0}^r (-1)^{r-j} M_{\overrightarrow{h}, j} = \sum_{j=|\{h_1,\ldots, h_{r-1}\}|}^r (-1)^{r-j} {r-|\{h_1,\ldots, h_{r-1}\}| \choose j-|\{h_1,\ldots, h_{r-1}\}|} = \sum_{j=0}^n (-1)^{n-j} {n \choose j} = 0
    \]
    where $n = r-|\{h_1,\ldots, h_{r-1}\}|$. The last equality is true, because $n \ge 1$. This completes step 1.
    Now for step 2 we need to show that 
    \[
        \Phi_{r-1}\circ \Psi_r'\left(\left\{\left(\sum_{y\in C} n_y [y],0\right),(0,a_2) \ldots, (0,a_r)\right\}_{k'/k}\right) = 0
    \]
    where $y$ is a $k'$-rational point whenever $n_y\ne 0$. 
    By definition we have that 
    \begin{align*}
        &\Psi_r'\left(\left\{\left(\sum_{y\in C} n_y [y],0\right),(0,a_2) \ldots, (0,a_r)\right\}_{k'/k}\right)\\
        &= \left(\sum_{y}n_{y}\sum_{j=0}^{r-1}(-1)^{r- 1-j}(\pi_{k'/k})_*\left(\sum_{2\le \nu_1< \ldots \ldots  < \nu_j \le r}[y\times_{k'}(a_{\nu_1}+\ldots + a_{\nu_j})]_{k'}\right)\right).
    \end{align*}
    
    \begin{rem}
        In the analogous part of the argument for the general case, the same idea will apply. The idea is that we use the fact that the cycles coming from the curves have degree $0$ along with multilinearity in $S_r$ to kill the symbols. 
    \end{rem}
    In what follows, we will denote $[y]_{k'} - [0]_{k'}$ by $y_{k'}$.
    By Proposition \ref{prop:Phi-Tr-commute}, we have
    \begin{align*}
        &\Phi_{r-1}\left(\sum_{y}n_{y}\sum_{j=0}^{r-1}(-1)^{r-1-j}(\pi_{k'/k})_*\left(\sum_{2\le \nu_1< \ldots < \nu_j \le r}[y\times_{k'} (a_{\nu_1}+\ldots + a_{\nu_j})]_{k'}\right)\right) \\
        &=\Tr_{k'/k}\Phi^{k'}_{r-1}\left(\sum_{y}n_{y}\sum_{j=0}^{r-1}(-1)^{r-1-j}\left(\sum_{2\le \nu_1< \ldots  < \nu_j \le r}[y\times_{k'} (a_{\nu_1}+\ldots + a_{\nu_j})]_{k'}\right)\right) \\ 
        &= \sum_{j=0}^{r-1}(-1)^{r-1-j}\sum_{2\le \nu_1< \ldots  < \nu_j \le r} \sum_{y}n_{y}^i\{\underbrace{(y_{k'}, a_{\nu_1}+\ldots + a_{\nu_j}), \ldots, (y_{k'}, a_{\nu_1}+\ldots + a_{\nu_j})}_{r-1}\}_{k'/k} \\
        &= \sum_{j=0}^{r-1}(-1)^{r-1-j} \sum_{2\le \nu_1< \ldots < \nu_j \le r}\\
        &\sum_{y}n_y\{\underbrace{(0, a_{\nu_1}+\ldots + a_{\nu_j})+(y_{k'}, 0) , \ldots, (0, a_{\nu_1}+\ldots + a_{\nu_j})+(y_{k'}, 0)}_{r-1}\}_{k'/k}.
    \end{align*}
    After expanding and reordering using the multilinearity and symmetry of $S_r(k;\underline{J_1}\times A)$, we can rewrite the above as 
    \begin{align*}
        &\sum_{j=0}^{r-1}(-1)^{r-1-j}\sum_{2\le \nu_1< \ldots < \nu_j \le r}\sum_{\ell=0}^{r-1}{r-1 \choose \ell}\\
        &\sum_{y}n_{y}\{\underbrace{(y_{k'}, 0), \ldots, (y_{k'}, 0)}_\ell, \underbrace{(0, a_{\nu_1}+\ldots + a_{\nu_j}), \ldots,(0, a_{\nu_1}+\ldots + a_{\nu_j})}_{r-1-\ell} \}_{k'/k}.
    \end{align*}
    Thus it suffices to show for each $0 \le \ell \le r-1$, 
    \begin{align*}
        &\sum_{j=0}^{r-1}(-1)^{r-1-j} \sum_{2\le \nu_1< \ldots < \nu_j \le r}\\
        &\sum_{y}n_y\{\underbrace{(y_{k'}, 0), \ldots, (y_{k'}, 0)}_\ell, \underbrace{(0, a_{\nu_1}+\ldots + a_{\nu_j}), \ldots,(0, a_{\nu_1}+\ldots + a_{\nu_j})}_{r-1-\ell} \}_{k'/k} = 0.
    \end{align*}
    In the case when $\ell=0$, the above terms in the sum do not depend on $y$, so the above sum is $0$ since $\sum_y n_y = 0$ by assumption. 
    \\So, it suffices to show for fixed $1 \le \ell \le r-1$, 
    \begin{align*}
        &\sum_{j=0}^{r-1}(-1)^{r-1-j} \sum_{2\le \nu_1< \ldots  < \nu_j \le r} \\
        & \{\underbrace{(y_{k'}, 0), \ldots, (y_{k'}, 0)}_\ell, \underbrace{(0, a_{\nu_1}+\ldots + a_{\nu_j}), \ldots,(0, a_{\nu_1}+\ldots + a_{\nu_j})}_{r-1-\ell} \}_{k'/k} = 0.
    \end{align*}
    However, now we have $r-1$ $a_i$ terms from the abelian variety and $< r-1$ spots in the symbol, so the same argument from step 1 applies (see Remark \ref{rem: Part 1 of landing in F^r}), and this sum is zero. This completes step 2.
\end{proof}
Now we provide the proof in the case when $d$ is any positive integer. The structure of the proof is the same, but there is heavier notation and added complexity in the parts of the proof which we indicate.

\begin{proof}[Proof of the general case]
    It suffices to prove the propostion for symbols of the form
    \begin{align*}
        &\left\{\left(0, \ldots, 0,\sum_{y_{i_1}} n_{y_{i_1}} [y_{i_1}],0 \ldots, ,0\right), \ldots, \left(0, \ldots, 0,\sum_{y_{i_s}} n_{y_{i_s}} [y_{i_s}], 0,\ldots,0\right),\right. \\
        &\left.(0,\ldots, a_1),\ldots, (0,\ldots, a_{r-s})\vphantom{\sum_{i=1}^\infty}\right\}_{k'/k}
    \end{align*}
    where each $y_{i_t}$ is a $k'$ rational point of $(C_t)_{k'}$ if $n_{y_{i_t}} \ne 0$ (see Remark \ref{rem: Reduction of proofs}). Because we will now only deal with rational points we will denote cycles of the form 
    \[ 
        [p_1\times_{k'} \ldots \times_{k'} y_{i_1}\times_{k'} \ldots\times_{k'} y_{i_s}\times_{k'} \ldots\times_{k'} p_d\times_{k'} a]_{k'}
    \]
    unambiguously by 
    \[ 
        [(p_1,\ldots, y_{i_1}, \ldots,y_{i_s}, \ldots, p_d, a)]_{k'}.
    \]
    Thus, by the alternative definition of the $\Psi_r'$ map (Proposition \ref{prop: psi alt def}) we need to show that 
    \begin{align*}
        & \sum_{y_{i_1}} \ldots \sum_{y_{i_s}} n_{y_{i_1}} \cdots n_{y_{i_s}} \sum_{j=0}^{r-s} (-1)^{r-s-j} \\
        &\Phi_{r-1}\left((\pi_{k'/k})_*\left(\sum_{1 \le \nu_1 <\ldots < \nu_j \le r} [(p_1,\ldots, y_{i_1}, \ldots,y_{i_s}, \ldots, p_d, a_{\nu_1}+\ldots +a_{\nu_j})]_{k'} \right)\right) = 0.
    \end{align*}
    By Proposition \ref{prop:Phi-Tr-commute}, we have 
    \begin{align*}
        &\Phi_{r-1}\left( (\pi_{k'/k})_*\left( [(p_1,\ldots, y_{i_1}, \ldots,y_{i_s}, \ldots, p_d, a_{\nu_1}+\ldots +a_{\nu_j})]_{k'}\vphantom{\sum_{i=1}} \right)\right) \\
        &= \Tr_{k'/k}\left(\Phi^{k'}_{r-1}\left( [(p_1,\ldots, y_{i_1}, \ldots,y_{i_s}, \ldots, p_d, a_{\nu_1}+\ldots +a_{\nu_j})]_{k'}\vphantom{\sum_{i=1}}\right)\vphantom{\sum_{i=1}}\right). \\
    \end{align*}
    Therefore, it suffices to assume that $k' = k$. 
    Denote $y_t := [y_t] - [p_t]$. Then by definition,
    \begin{align*}
        &\Phi_{r-1}\left([(p_1, \ldots, y_{i_1}, \ldots, y_{i_s}, \ldots, p_d, a_{\nu_1}+ \ldots + a_{\nu_j})] \vphantom{\sum_{i=1}}\right)\\
        &= \left\{\left(0, \ldots, 0, y_{i_1}, \ldots, y_{i_s}, \ldots, 0, a_{\nu_1}+\ldots + a_{\nu_j}\vphantom{\sum_{i=1}^\infty}\right), \ldots \right.\\
        &\left.\left(0, \ldots, 0, y_{i_1}, \ldots, y_{i_s}, \ldots, 0, a_{\nu_1}+\ldots + a_{\nu_j}\vphantom{\sum_{i=1}^\infty}\right)\right\}_{k/k}.
    \end{align*}
    So the above becomes 
    \begin{align*}
        & \sum_{y_{i_1}} \ldots \sum_{y_{i_s}} n_{y_{i_1}} \cdots n_{y_{i_s}} \sum_{j=0}^{r-s} (-1)^{r-s-j}\sum_{1 \le \nu_1 < \ldots < \nu_j \le r-s} \\
        &\left\{\left(0, \ldots, 0, y_{i_1}, \ldots, y_{i_s}, \ldots, 0, a_{\nu_1}+\ldots + a_{\nu_j}\vphantom{\sum_{i=1}^\infty}\right), \ldots \right.\\
        &\left.\left(0, \ldots, 0, y_{i_1}, \ldots, y_{i_s}, \ldots, 0, a_{\nu_1}+\ldots + a_{\nu_j}\vphantom{\sum_{i=1}^\infty}\right)\right\}_{k/k}.
    \end{align*}
    Let $N_{i_1}, \ldots, N_{i_s} \in \Z_{\ge 0}$ such that $N_{i_1} + \ldots + N_{i_s} \le r-1$. Then denote 
    \begin{align*}
        X(N_{i_1}, \ldots, N_{i_s},a_{\nu_1}, \ldots, a_{\nu_j}) &:= \{\underbrace{(0, \ldots, 0, y_{i_1}, 0 \ldots, 0), \ldots, (0, \ldots, 0, y_{i_1}, 0 \ldots, 0)}_{N_{i_1}},\ldots \\
        &\underbrace{(0, \ldots, 0, y_{i_s}, 0 \ldots, 0), \ldots, (0, \ldots, 0, y_{i_s}, 0 \ldots, 0)}_{N_{i_s}}, \\
        & \underbrace{(0, \ldots, a_{\nu_1}+\ldots + a_{\nu_j}), \ldots, (0, \ldots, a_{\nu_1}+\ldots + a_{\nu_j})}_{r-1 - N_{i_1}- \ldots - N_{i_s}}\}_{k/k}.
    \end{align*}
    Using this notation, by multilinearity and symmetry of symbols, the above becomes 
    \begin{align*}
        & \sum_{y_{i_1}} \ldots \sum_{y_{i_s}} n_{y_{i_1}} \cdots n_{y_{i_s}} \sum_{j=0}^{r-s} (-1)^{r-s-j}\sum_{1 \le \nu_1 < \ldots < \nu_j \le r-s} \\
        &\sum_{N_{i_1}+\ldots + N_{i_s}\le r-1}\beta_{N_{i_1},\ldots, N_{i_s}} X(N_{i_1}, \ldots, N_{i_s},a_{\nu_1}, \ldots, a_{\nu_j}),
    \end{align*}
    where $\beta_{N_{i_1},\ldots, N_{i_s}}:= \binom{r-1}{N_{i_1}}\binom{r-1-N_{i_1}}{N_{i_2}}\cdots \binom{r-1-N_{i_1}-\ldots - N_{i_{s-1}}}{N_{i_s}}$
    
    We will proceed in two steps. In the first step we consider the case when $N_{i_1}+\ldots + N_{i_s}\ge s$. Specifically, we prove that for fixed $y_{i_1}, \ldots, y_{i_s}$ and fixed $N_{i_1}+\ldots + N_{i_s}\ge s$, we have
    \begin{align*}
        &\sum_{j=0}^{r-s} (-1)^{r-s-j} \sum_{1 \le \nu_1 < \ldots < \nu_j \le r-s} X(N_{i_1},\ldots, N_{i_s}, a_{\nu_1},\ldots, a_{\nu_j}) = 0.
    \end{align*}
    In the second step we consider the case when $N_{i_1}+\ldots + N_{i_s}< s$. Specifically, we prove that for fixed $j$, $1\le \nu_1 < \ldots < \nu_j \le r-s$, and $N_{i_1}+\ldots + N_{i_s}< s$
    \begin{align*}
        &\sum_{y_{i_1}} \ldots \sum_{y_{i_s}} n_{y_{i_1}} \cdots n_{y_{i_s}} X(N_{i_1},\ldots, N_{i_s}, a_{\nu_1},\ldots, a_{\nu_j}) = 0.
    \end{align*}
    \begin{rem}
        These two cases mirror those of the proof when $d=1$. The case of $N_{i_1}+\ldots +N_{i_s}\ge s$ is the same argument that is used in the first part of the proof of one curve. We use a counting argument to show that all the relevant symbols vanish. The case of $N_{i_1}+\ldots +N_{i_s}< s$ is analogous to step 2 of the proof when $d=1$. Here we again use multilinearity and the fact that the cycles coming from the curve are degree $0$ to kill all of the symbols.
    \end{rem}
    Let us first consider the case when $N_{i_1}+\ldots +N_{i_s} \ge s$. For $\overrightarrow{\ell} \in \{1,\ldots, r-s\}^{r-1-(N_{i_1} + \ldots + N_{i_s})}$ define symbols
    \begin{align*}
        X(N_{i_1}, \ldots, N_{i_s},a_1,\ldots, a_{r-s})_{\overrightarrow{\ell}} := &\{\underbrace{(0, \ldots, 0, y_{i_1}, 0 \ldots, 0), \ldots, (0, \ldots, 0, y_{i_1}, 0 \ldots, 0)}_{N_{i_1}},\ldots \\
        &\underbrace{(0, \ldots, 0, y_{i_s}, 0 \ldots, 0), \ldots, (0, \ldots, 0, y_{i_s}, 0 \ldots, 0)}_{N_{i_s}}, \\
        & (0, \ldots,0, a_{\ell_1}), \ldots, (0, \ldots,0, a_{\ell_{r-1-(N_{i_1} + \ldots + N_{i_s})}})\}_{k/k}.
    \end{align*}
    Note that by multilinearity of symbols in $S_r(k; \underline{J_1}\times\ldots \times \underline{J_d}\times A)$, we have that 
    \[
        X(N_{i_1}, \ldots, N_{i_s},a_1,\ldots, a_r) = \sum_{\overrightarrow{\ell} \in \{\nu_1,\ldots, \nu_j\}^{r-1-(N_{i_1}+\ldots + N_{i_s})}}X(N_{i_1}, \ldots, N_{i_s},a_1,\ldots, a_{r-s})_{\overrightarrow{\ell}}.
    \]
    Thus, we have that 
    \begin{align*}
        &\sum_{j=0}^{r-s} (-1)^{r-s-j} \sum_{1 \le \nu_1 < \ldots < \nu_j \le r-s} X(N_{i_1},\ldots, N_{i_s}, a_{\nu_1},\ldots, a_{\nu_j}) \\
        &= \sum_{j=0}^{r-s} (-1)^{r-s-j} \sum_{1 \le \nu_1 < \ldots < \nu_j \le r-s}\sum_{\overrightarrow{\ell} \in \{\nu_1,\ldots, \nu_j\}^{r-1-(N_{i_1}+\ldots + N_{i_s})}}X(N_{i_1}, \ldots, N_{i_s},a_1,\ldots, a_{r-s})_{\overrightarrow{\ell}}.
    \end{align*}
    However, there will be some cancellation, because the same vector $\overrightarrow{\ell}$ can appear for two different sets $\{\nu_1,\ldots, \nu_j\}$ and $\{\nu_1,\ldots, \nu_{j'}\}$. To account for this, define the constants 
    \[
        M_{\overrightarrow{\ell}, j} := \left|\{Y \subset \{1,\ldots, r-s\} : |Y| = j, \overrightarrow{\ell} \in Y^{r-1-(N_{i_1}+\ldots + N_{i_s})}\}\right|.
    \]
    Then 
    \begin{align*}
        &\sum_{j=0}^{r-s}(-1)^{r-s-j} \sum_{1 \le \nu_1 < \ldots < \nu_{j} \le r} X(N_{i_1}, \ldots, N_{i_s}, a_{\nu_1},\ldots, a_{\nu_j}) \\
        &= \sum_{\overrightarrow{\ell} \in \{1,\ldots, r-s\}^{r-1-(N_{i_1}+\ldots + N_{i_s})}}\sum_{j=0}^{r-s}(-1)^{r-s-j} M_{\overrightarrow{\ell}, j} X(N_{i_1}, \ldots, N_{i_s}, a_1,\ldots, a_r)_{\overrightarrow{\ell}}.
    \end{align*}
    
    Thus it suffices to show that for fixed $\overrightarrow{\ell}$ 
    \[
        \sum_{j=0}^{r-s}(-1)^{r-s-j} M_{\overrightarrow{\ell}, j} = 0.
    \]
    We have 
    \[
        M_{\overrightarrow{\ell}, j} = {r-s-|\overrightarrow{\ell}| \choose j - |\overrightarrow{\ell}|},
    \]
    so 
    \[
        \sum_{j=0}^{r-s}(-1)^{r-s-j} M_{\overrightarrow{\ell}, j} = \sum_{j=0}^{r-s}(-1)^{r-s-j} {r-s-|\overrightarrow{\ell}| \choose j - |\overrightarrow{\ell}|} = \sum_{k=0}^n (-1)^{n-k}{n \choose k}
    \]
    where $n = r-s-|\overrightarrow{\ell}|$ and $k = j - |\overrightarrow{\ell}|$. The last sum is 0, because $|\overrightarrow{\ell}| < r-s$ by assumption, so $n$ is positive. This completes step $1$.

    \begin{rem}[Important]\label{rem: proof of corollary}
        Suppose that we had applied $\Phi_r$ instead of $\Phi_{r-1}$. Then in the above we would have ended up with the sum 
        \[
            \sum_{\overrightarrow{\ell} \in \{1,\ldots, r-s\}^{r-(N_{i_1}+\ldots + N_{i_s})}}\sum_{j=0}^{r-s}(-1)^{r-s-j} M_{\overrightarrow{\ell}, j} X(N_{i_1}, \ldots, N_{i_s}, a_1,\ldots, a_r)_{\overrightarrow{\ell}}
        \]
        rather than 
        \[
            \sum_{\overrightarrow{\ell} \in \{1,\ldots, r-s\}^{r-1-(N_{i_1}+\ldots + N_{i_s})}}\sum_{j=0}^{r-s}(-1)^{r-s-j} M_{\overrightarrow{\ell}, j} X(N_{i_1}, \ldots, N_{i_s}, a_1,\ldots, a_r)_{\overrightarrow{\ell}}.
        \]
        In this case we could have had $|\overrightarrow{\ell}| = r-s$, which would cause
        \[
            \sum_{j=0}^{r-s}(-1)^{r-s-j}M_{\overrightarrow{\ell},j} = 1. 
        \]
        Since all $X(N_{i_1}, \ldots, N_{i_s}, a_1,\ldots, a_r)_{\overrightarrow{\ell}}$ are equal as long as $|\overrightarrow{\ell}| = r-s$ by symmetry of symbols, we would get that 
        \begin{align*}
        &\sum_{j=0}^{r-s} (-1)^{r-s-j} \sum_{1 \le \nu_1 < \ldots < \nu_j \le r-s} X(N_{i_1},\ldots, N_{i_s}, a_{\nu_1},\ldots, a_{\nu_j}) \\
        &= (r-s)!\{(0, \ldots, 0, y_{i_1}, 0 \ldots, 0),\ldots (0, \ldots, 0, y_{i_s}, 0 \ldots, 0), (0, \ldots, a_1), \ldots, (0, \ldots, a_{r-s})\}_{k/k}.
    \end{align*}
        
    \end{rem}

    Now let us consider the case when $N_{i_1}+\ldots +N_{i_s} < s$. This implies that for some $1\le t \le s$, $N_{i_t}=0$. In other words, the symbol 
    \[
        X(N_{i_1}, \ldots, N_{i_s}, a_{\nu_1}, \ldots, a_{\nu_j})
    \]
    is independent of $y_{i_t}$. Therefore, 
    \begin{align*}
        &\sum_{y_{i_1}} \ldots \sum_{y_{i_s}} n_{y_{i_1}} \cdots n_{y_{i_s}} X(N_{i_1},\ldots, N_{i_\ell}, a_{\nu_1},\ldots, a_{\nu_j}) \\
        &= \sum_{y_{i_t}}n_{y_{i_t}}\left(\sum_{y_{i_1}} \ldots \sum_{y_{i_s}} n_{y_{i_1}} \cdots n_{y_{i_s}} X(N_{i_1},\ldots, N_{i_\ell}, a_{\nu_1},\ldots, a_{\nu_j})\right) = 0
    \end{align*}
    because $\sum_{y_{i_t}}n_{y_{i_t}} = 0$ by assumption. This completes step $2$. 
\end{proof}

As a corollary of the proof we have the following:
\begin{cor}\label{cor: phi circ psi}
    The composition $\Phi_r\circ \Psi_r'$ is given by multiplication by $r!$. 
\end{cor}

\begin{rem}
    Note that the domain of $\Phi_r\circ \Psi_r'$ is 
    \[
        \bigoplus_{k'/k}\prod_{i=1}^r J_1(k')\oplus \cdots \oplus J_d(k')\oplus A(k')
    \]
    while the codomain is
    \[
        S_r(k; \underline{J_1}\times\ldots \times \underline{J_d}\times A).
    \]
    So the precise statement of the corollary is that for $\alpha \in  \bigoplus_{k'/k}\prod_{i=1}^r J_1(k')\oplus \cdots \oplus J_d(k')\oplus A(k')$, $\Phi_r\circ \Psi_r'(\alpha) = r!q(\alpha)$ where 
    \[
        q: \bigoplus_{k'/k}\prod_{i=1}^r J_1(k')\oplus \cdots \oplus J_d(k')\oplus A(k') \to S_r(k; \underline{J_1}\times\ldots \times \underline{J_d}\times A)
    \]
    is the natural quotient. 
\end{rem}

\begin{proof}
    Using the notation of the above proof, we computed that 
    \begin{align*}
        \Phi_r\circ \Psi_r'&\left\{\left(0, \ldots, 0,\sum_{y_{i_1}} n_{y_{i_1}} [y_{i_1}],0 \ldots, ,0\right), \ldots, \left(0, \ldots, 0,\sum_{y_{i_s}} n_{y_{i_s}} [y_{i_s}], 0,\ldots,0\right),\right. \\
        &\left.(0,\ldots, a_1),\ldots, (0,\ldots, a_{r-s})\vphantom{\sum_{i=1}^\infty}\right\}_{k'/k} \\
        &= \sum_{y_{i_1}} \ldots \sum_{y_{i_s}} n_{y_{i_1}} \cdots n_{y_{i_s}} \sum_{j=0}^{r-s} (-1)^{r-s-j}\sum_{1 \le \nu_1 < \ldots < \nu_j \le r-s} \\
        &\sum_{N_{i_1}+\ldots + N_{i_s}\le r}\beta_{N_{i_1},\ldots, N_{i_s}}X(N_{i_1}, \ldots, N_{i_s},a_{\nu_1}, \ldots, a_{\nu_j}).
    \end{align*}
    However, if $N_{i_1}+\ldots + N_{i_s} > s$, then we fall into step 1 in the above proof and if any $N_{i_1}, \ldots, N_{i_s}$ are $0$, then we fall into step $2$. Hence, we must have $N_{i_1} = \ldots = N_{i_s} = 1$. In this case $\beta_{1,\ldots ,1} = \frac{r!}{(r-s)!}$. Also by the above proof (see Important Remark \ref{rem: proof of corollary}), in this case
    \begin{align*}
        &\sum_{j=0}^{r-s} (-1)^{r-s-j}\sum_{1 \le \nu_1 < \ldots < \nu_j \le r-s} X(N_{i_1}, \ldots, N_{i_s},a_{\nu_1}, \ldots, a_{\nu_j}) \\
        &= (r-s)!\{(0, \ldots, 0, y_{i_1, k'}, 0 \ldots, 0),\ldots (0, \ldots, 0, y_{i_s, k'}, 0 \ldots, 0), (0, \ldots, a_1), \ldots, (0, \ldots, a_{r-s})\}_{k'/k}.
    \end{align*}
    Putting this together, we get 
    \begin{align*}
        \Phi_r\circ \Psi_r'&\left\{\left(0, \ldots, 0,\sum_{y_{i_1}} n_{y_{i_1}} [y_{i_1}],0 \ldots, ,0\right), \ldots, \left(0, \ldots, 0,\sum_{y_{i_s}} n_{y_{i_s}} [y_{i_s}], 0,\ldots,0\right),\right. \\
        &\left.(0,\ldots, a_1),\ldots, (0,\ldots, a_{r-s})\vphantom{\sum_{i=1}^\infty}\right\}_{k'/k} \\
        &= \sum_{y_{i_1}} \ldots \sum_{y_{i_s}} n_{y_{i_1}} \cdots n_{y_{i_s}}  \frac{r!}{(r-s)!}(r-s)!\\
        &\{(0, \ldots, 0, y_{i_1, k'}, 0 \ldots, 0),\ldots (0, \ldots, 0, y_{i_s, k'}, 0 \ldots, 0), (0, \ldots, a_1), \ldots, (0, \ldots, a_{r-s})\}_{k'/k}\\
        &=r!\{(0, \ldots, 0, \sum_{y_{i_1}}n_{y_{i_1}}[y_{i_1}], 0 \ldots, 0),\ldots (0, \ldots, 0, \sum_{y_{i_s}}n_{y_{i_s}}[y_{i_s}], 0 \ldots, 0), \\
        &(0, \ldots, a_1), \ldots, (0, \ldots, a_{r-s})\}_{k'/k}
    \end{align*}
    where the last equality follows by multilinearity of symbols.
\end{proof}

\begin{defn}
    The map 
    \[
        \Psi_r: \bigoplus_{k'/k} \prod_{i=1}^r J_1(k') \oplus \cdots \oplus J_d(k') \oplus A(k') \to \frac{F^r\ch_0(X)}{F^{r+1}\ch_0(X)}
    \]
    is defined as the composite
    \[
        \bigoplus_{k'/k} \prod_{i=1}^r J_1(k') \oplus \cdots \oplus J_d(k') \oplus A(k') \xrightarrow{\Psi'_r} F^r\ch_0(X) \to \frac{F^r\ch_0(X)}{F^{r+1}\ch_0(X)}.
    \]
\end{defn}

Now we are able to complete the proofs of the key lemmas.

\begingroup
\def\thelem{\ref{lem: key-lem-Psi}}
\begin{lem}
    The map $\Psi_r$ induces a group homomorphism 
    \[
        S_r(k; \underline{J_1}\times \cdots \times \underline{J_d}\times A) \to F^r\ch_0(X)/F^{r+1}\ch_0(X).
    \]
\end{lem}
\addtocounter{lem}{-1}
\endgroup
\begin{proof}
    Let $H$ be the subgroup of $\bigoplus_{k'/k}\prod_{i=1}^r J_1(k') \oplus \cdots \oplus J_d(k') \oplus A(k')$ such that the quotient by $H$ gives $S_r(k; \underline{J_1}\times \cdots \times \underline{J_d}\times A)$. Let $x \in H$. Then by Corollary \ref{cor: phi circ psi} $\Phi_r(\Psi_r(x)) = r!x$, which is in $H$. Therefore, $\Phi_r(\Psi_r(x)) = 0$, so $\Psi_r(x) = 0$, because $\Phi_r$ is injective on $F^r\ch_0(X)/F^{r+1}\ch_0(X)$.
\end{proof}

Note that Corollary \ref{cor: phi circ psi} and Lemma $\ref{lem: key-lem-Psi}$ imply Lemma \ref{lem: key-lem-comp}. 

\raggedbottom
\pagebreak

\section{Finiteness of the Filtration}\label{section: finiteness of filt}
In this section we show that $F^{r}(X)\otimes \mathbb{Q} = 0$ for all $r > \text{dim}(X) = d + g$ where $g$ is the dimension of $A$. First, we will need some preparation regarding specialization maps.

\subsection{Specialization Maps}
Let $X$ be a smooth projective variety over $k$, and let $K$ be a function field in one variable over $k$. The for each place $v$ of $K$ there exist specialization maps 
\[
    s_v: \ch_0(X_K) \to \ch_0(X_{k(v)})
\]
where $k(v)$ is the residue field of $v$ defined as follows. Let $\mathcal{O}_v$ be the ring of integers of $v$. Then $X_K \to X_{\mathcal{O}_v}$ is an open embedding and $X_{k(v)} \to X_{\mathcal{O}_v}$ is a closed embedding. Then for $x \in \ch_0(X_K)$, $s_v(x)$ is $\overline{x} \cap X_{k(v)}$ where $\overline{x}$ is the closure of $x$ in $X_{\mathcal{O}_v}$. 

In the case that $X$ is a curve with a rational point and $J$ is its Jacobian, then under the natural isomorphism $\Pic^0(X_E) \cong J(E)$ for an extension $E/k$, the specialization map on Chow groups agrees with the specialization map on points of the abelian variety defined in Section \ref{section: Somekawa-Kgps}. Moreover, if $x \in X(K)$, then $s_v(x) \in X(k(v))$, and the cycle class it defines in $\ch_0(X_{k(v)})$ is $s_v([x]_K)$ where $[x]_K$ is viewed as an element of $\ch_0(X_K)$.  

\begin{defn}
    Recall the map 
    \[
        \Psi'_r: \bigoplus_{k'/k}\prod_{i=1}^r J_1(k')\oplus \cdots \oplus J_d(k') \oplus A(k') \to \ch_0(X)
    \]
    from Definiton \ref{def:Psi-map}. Then define $F^r_{\Psi}\ch_0(X) \subset \ch_0(X)$ to be the subgroup generated by the images of $\Psi_j'$ for all $j\ge r$. That is
    \[
        F^r_{\Psi}\ch_0(X) := \left\langle\bigcup_{j \ge r} \mathrm{im}(\Psi_r') \right\rangle
    \]
\end{defn}

Note that $F^r_{\Psi}\ch_0(X)$ defines a descending filtration on $\ch_0(X)$. By Proposition \ref{prop:Psi-in-Fr},  $F^r_{\Psi}\ch_0(X) \subset F^r\ch_0(X)$.

The groups $F^r_{\Psi}\ch_0(X)$ are analogues of Pontryagin powers of the degree $0$ subgroup of $\ch_0(A)$ for an abelian variety $A$. In the case when $d = 0$, these two groups coincide. Importantly for us, if $A$ is an abelian variety of dimension $g$, and $I$ is the degree $0$ subgroup of $\ch_0(A)$, then 
\[
    I^{\odot g+1}\otimes \mathbb{Q} = 0
\]
in $\ch_0(A)\otimes \mathbb{Q}$. This fact was proven by Bloch \cite{bloch1976some} over algebraically closed fields. Beauville proved it over $\mathbb{C}$ \cite{beauville1986anneau} using different methods, and Deninger and Murre \cite{murre1991motivic} extended these techniques to arbitrary base fields. 

\begin{lem}\label{lemma: G^r vanishes}
    For every $r > \text{dim}(X)$, $F^r_\Psi(X)\otimes \mathbb{Q} = 0$. 
\end{lem}

\begin{proof}
    By the definition of $F^r_\Psi(X)$ and $\Psi_r'$, it suffices to show that for $r > dim(X)$ and fixed $0\le s \le d$, $1\le j_1 < \ldots < j_s \le r$, $a_1,\ldots, a_r \in A(k)$
    \[
        ([a_1]_{k'}-[0]_{k'})\odot \cdots \odot \widehat{([a_{j_1}]_{k'})}\odot \cdots  \odot \widehat{([a_{j_s}]_{k'})} \odot \cdots \odot ([a_r]_{k'}-[0]_{k'}) = 0.
    \]
    This is true by the above since $r- s \ge r-d > \text{dim}(A)$. 
\end{proof}
Therefore, it remains to verify the following proposition

\begin{prop}\label{prop: Fr and FrPsi agree}
    Over an arbitrary field, $F^r\ch_0(X)\otimes \mathbb{Q} = F^r_\Psi\ch_0(X)\otimes \mathbb{Q}$.  
\end{prop}

Before, we prove this proposition, we need some preliminaries. First note that by Proposition \ref{prop:Psi-in-Fr}, $\Phi_r$ induces a well-defined map
\[
    \overline{\Phi}_r: F^r_{\Psi}\ch_0(X)/F_{\Psi}^{r+1}\ch_0(X) \to S_r(k; \underline{J_1}\times \cdots \times \underline{J_d}\times A).
\]
Now denote by $\overline{\Psi}_r$ the composition
\[
    \bigoplus_{k'/k}\prod_{i=1}^rJ_1(k')\oplus \ldots J_d(k') \oplus A(k') \xrightarrow{\Psi_r'} F^{r}_{\Psi}\ch_0(X) \to F^r_{\Psi}\ch_0(X)/F^{r+1}_{\Psi}\ch_0(X).
\]

\begin{lem}
    Let $k = \overline{k}$ be algebraically closed. Then $\overline{\Psi}_r$ induces a well-define group homomorphism 
    \[
        S_r(k; \underline{J_1}\times \cdots \times \underline{J_d}\times A) \to F^r_{\Psi}\ch_0(X)/F^{r+1}_{\Psi}\ch_0(X).
    \]
    Moreover, the composition
    \[
        F^r_{\Psi}\ch_0(X)/F^{r+1}_{\Psi}\ch_0(X)\xrightarrow{\overline{\Phi}_r}S_r(k; \underline{J_1}\times \cdots \times \underline{J_d}\times A) \xrightarrow{\overline{\Psi}_r} F^r_{\Psi}\ch_0(X)/F^{r+1}_{\Psi}\ch_0(X)
    \]
    is multiplication by $r!$. 
\end{lem}

\begin{proof}
    Consider the map
    \[
        \Psi_r' : \bigoplus_{k'/k}\prod_{i=1}^rJ_1(k')\oplus \cdots \oplus J_d(k') \oplus A(k') \to F^r_{\Psi}\ch_0(X).
    \]
    Our first step is to show that this induces a group homomorphism 
    \[
        \overline{\Psi}_r : S_r(k; \underline{J_1}\times \ldots \times \underline{J_d} \times A) \to F^r_{\Psi}\ch_0(X)/F^{r+1}_\Psi\ch_0(X).
    \]
    Let us first show that the above map is multilinear. By Proposition \ref{prop: partial multilinearity}, it suffices to show that 
    \begin{align*}
        &\overline{\Psi}_r\left(\{(0, \ldots, 0, a_1+\tilde{a}_1), (z_1^2, \ldots, z_d^2, a_2),\ldots, (z_1^r,\ldots, z_d^r, a_r)\}_{k'/k}\right) = \\
        & \overline{\Psi}_r\left(\{(0, \ldots, 0, a_1),(z_1^2, \ldots, z_d^2, a_2),\ldots, (z_1^r,\ldots, z_d^r, a_r)\}_{k'/k}\right) \\
        & +\overline{\Psi}_r\left(\{(0, \ldots, 0, \tilde{a}_1)(z_1^2, \ldots, z_d^2, a_2),\ldots, (z_1^r,\ldots, z_d^r, a_r)\}_{k'/k}\right).
    \end{align*}
    For this it suffices to observe that 
    \begin{align*}
        &\overline{\Psi}_{r+1}\left(\{(0, \ldots, 0, a_1), (0, \ldots, 0, \tilde{a}_1), (z_1^2, \ldots, z_d^2, a_2) \ldots, \ldots, (z_1^r,\ldots, z_d^r, a_r)\}_{k'/k}\right) = \\
        &\overline{\Psi}_r\left(\{(0, \ldots, 0, a_1+\tilde{a}_1), (z_1^2, \ldots, z_d^2, a_2) \ldots, \ldots, (z_1^r,\ldots, z_d^r, a_r)\}_{k'/k}\right) \\
        & -\overline{\Psi}_r\left(\{(0, \ldots, 0, a_1),(z_1^2, \ldots, z_d^2, a_2),\ldots, (z_1^r,\ldots, z_d^r, a_r)\}_{k'/k}\right) \\
        & -\overline{\Psi}_r\left(\{(0, \ldots, 0, \tilde{a}_1)(z_1^2, \ldots, z_d^2, a_2),\ldots, (z_1^r,\ldots, z_d^r, a_r)\}_{k'/k}\right).
    \end{align*}
    Thus, $\overline{\Psi}_r$ induces a group homomorphism
    \[
        \overline{\Psi}_r : \bigoplus_{k'/k}\bigotimes_{i=1}^rJ_1(k')\oplus \ldots J_d(k') \oplus A(k') \to F^r_{\Psi}\ch_0(X)/F^{r+1}_\Psi\ch_0(X).
    \]
    We want to show that the map $\overline{\Psi}_r$ induces a well-defined group homomorphism from $S_r(k; \underline{J_1}\times\ldots\times\underline{J_d}\times A)$. For the projection formula, there is nothing to show as $k$ is algebraically closed. Next, we need to show that $\overline{\Psi}_r$ factors through the quotient by Weil relations. Let $K$ be a function field in one variable, and for each place $v$ of $K$ let 
    \[
        s_v^i: J_i(K)\to J_i(k_v)
    \]
    and 
    \[
        s_v^A: A(K) \to A(k)
    \]
    be the specialization maps. By Propositions \ref{prop:partial proj form} and \ref{prop: partial multilinearity}, it suffices to show that for $f \in K^\times$, $z_{i_t} \in J_{i_t}(K)$, and $a_j \in A(k)$
    \begin{align*}
        &\Psi_r'\sum_{v \text{ place of } K/k} \ord_v f\{(0, \ldots, 0, s_v^{i_1}(z_{i_1}), 0 \ldots, 0), \ldots, \left(0, \ldots, 0,s_v^{i_s}(z_{i_s}), 0 \ldots, 0\right), \\
        & (0,\ldots,0, s_v^1(a_1)), \ldots, (0,\ldots, 0, s_v^A(a_{r-s}))\}_{k(v)/k} =0.
    \end{align*}
    We have, 
    \begin{align*}
        &\Psi_r'\sum_{v \in K/k} \ord_v f\{(0, \ldots, 0, s_v^{i_1}(z_{i_1}), 0 \ldots, 0), \ldots, \left(0, \ldots, 0,s_v^{i_s}(z_{i_s}), 0 \ldots, 0\right), \\
        & (0,\ldots,0, s_v^1(a_1)), \ldots, (0,\ldots, 0, s_v^A(a_{r-s}))\}_{k(v)/k} \\
        &= \sum_{v \text{ place of } K/k} \ord_v f(\pi_{k(v)/k})_*(W)
    \end{align*}
    where 
    \begin{align*}
        &W = \left(pr_{1,k(v)}^*([p_1]_{k(v)})\cdots pr_{i_1, k(v)}^*(s_v^{i_1}z_{i_1})\cdots pr_{i_s,k(v)}^*(s_v^{i_s}z_{i_s})\cdots pr_{d,k(v)}^*([p_d]_{k(v)})\right) \\
        &\cdot pr_{A,k(v)}^*\left(\bigodot_{j \in \{1,\ldots r-s\}}([s_v^A(a_j)]_{k(v)} - [0]_{k(v)})\right).
    \end{align*}
    But by the remarks above,
    \[
        \bigodot_{j \in \{1,\ldots r-s\}}([s_v^A(a_j)]_{k(v)} - [0]_{k(v)}) = s_v^A\left(\bigodot_{j \in \{1,\ldots r-s\}}([a_j]_{K}-[0]_{K})\right)
    \]
    where the specialization map on the right is now taken to be the specialization map on Chow groups. By Proposition 20.3 and Corollary 20.3 of \cite{fulton2013intersection}, the above becomes 
    \[
        \sum_{v\text{ place of } K/k} \ord_v f (\pi_{k(v)/k})_*(s_v^X(Q))
    \]
    where $s_v^X$ is the specialization map on $X$ and 
    \begin{align*}
        &Q = \left(pr_{1,K}^*([p_1]_{K})\cdots pr_{i_1, K}^*(z_{i_1})\cdots pr_{i_s,K}^*(z_{i_s})\cdots pr_{d,K}^*([p_d]_{K})\right) \\
        &\cdot pr_{A,K}^*\left(\bigodot_{j \in \{1,\ldots r-s\}}([a_j]_{K} - [0]_{K})\right).
    \end{align*}
    And this is $0$ by (2.2.1) of \cite{raskind2000milnor}.
    
    The fact $\overline{\Psi}_r$ is invariant under the action of the symmetric group is clear from the definition. Thus, we have a map
    \[
        \overline{\Psi}_r : S_r(k; J_1\times \cdots \times J_d\times A) \to F^r_{\Psi}\ch_0(X)/F^{r+1}_\Psi\ch_0(X).
    \]
    It remains to show that $\overline{\Psi}_r$ factors through $S_r(k; \underline{J_1}\times \cdots \times \underline{J_d}\times A)$. Since $\overline{\Psi}_r$ is multilinear and invariant under the action of the symmetric group, this amounts to checking that for each $1\le t\le d$
    \[
        \overline{\Psi}_r\left(\{(0,\ldots, 0, z_t^1,0,\ldots, 0),(0,\ldots, 0, z_t^2,0,\ldots, 0),(z_1^3, \ldots, z_d^3,a_3) ,\ldots, (z_1^r, \ldots, z_d^r, a_r)\}\right) =0. 
    \]
    By the definition $\overline{\Psi}_r$, the intersection product of each term in the sum will be $0$, so the above is $0$. So we have the desired map.
    
    Now we verify the second statement of the lemma. By Corollary \ref{cor: phi circ psi}, $\Psi_r'\circ\Phi_r$ is multiplication by $r!$. Let $\alpha \in F^r_\Psi\ch_0(X)/F^{r+1}_\Psi\ch_0(X)$. Then a representative of this coset can be taken to $\Psi_r'(\beta)$ for some $\beta \in S_r(k;\underline{J_1}\times\ldots\times\underline{J_d}\times A)$. Therefore, $\overline{\Psi}_r\circ\overline{\Phi}_r(\alpha) = \overline{\Psi}_r\circ\overline{\Phi}_r(\Psi_r'(\beta)) = \overline{\Psi}_r(r!\beta) = r!\alpha$
\end{proof}

\begin{lem}
    Let $k = \overline{k}$ be an algebraically closed field. Then $F^r\ch_0(X)\otimes \Q = F^r_{\Psi}\ch_0(X)\otimes \Q$. 
\end{lem}
\begin{proof}
     Now we show that $F^r\ch_0(X)\otimes \Q = F^r_\Psi\ch_0(X)\otimes \Q$ by induction on r. In the case that $r = 1$, $F^1\ch_0(X)$ is by definition the set of degree $0$ elements of $\ch_0(X)$. This group is generated by elements of the form $[(x_1,\ldots, x_d, a)] - [(p_1, \ldots, p_d, 0)]$ with $x_i \in C_i(k)$ and $a \in A(k)$ (recall that $k$ is algebraically closed). We show that this is an element of $F^1_{\Psi}\ch_0(X)$ by induction on the number of $x_i$ which are not equal to $p_i$. For the base case assume that $x_i = p_i$ for $1\le i \le d$. Then we have that 
     \[
        \Psi_1'(\{(0,\ldots, 0, a)\}_{k/k}) = [(p_1, \ldots, p_d, a)] - [(p_1, \ldots, p_d, 0)].
     \]
     Now assume that for some $1 \le t \le d$, and for some $1\le i_1 < \ldots < i_t \le d$ $x_{i_1} = p_{i_1}, \ldots x_{i_t} = p_{i_t}$. Then writing out the definition of 
     \[
        \Psi_{t+1}'\left(\{(0,\ldots, [x_{i_1}]-[p_{i_1}],0, \ldots, 0), \ldots, (0,\ldots, [x_{i_t}]-[p_{i_t}],0, \ldots, 0), (0,\ldots, 0, a)\}_{k/k}\right)
     \]
     we can see that there is a cycle of the form $[(p_1,\ldots, x_{i_1}, \ldots, x_{i_t}, \ldots, p_d, a)] - [(p_1, \ldots, p_d, 0)]$ and for all the terms in the remaining degree $0$ zero-cycle there are $<t$ components of the curve which are different from the $p_i$'s. So we conclude by the inductive hypothesis. 
     
    Now assume that $F^r\ch_0(X)\otimes \mathbb{Q} = F^r_{\Psi}\ch_0(X) \otimes \mathbb{Q}$. Then consider the commutative diagram
    \[
    \begin{tikzcd}
        F^r\ch_0(X)/F^{r+1}\ch_0(X) \arrow{r}{\Phi_r} & S_r(k; \underline{J_1}\times\ldots\times\underline{J_d}\times A) \\
        F^r_{\Psi}\ch_0(X)/F^{r+1}_{\Psi}\ch_0(X). \arrow{u}{q} \arrow{ur}{\overline{\Phi}_r} & 
    \end{tikzcd}
    \]
    where $q$ is the natural quotient. By the inductive hypothesis, it suffices to show that $q$ is injective after tensoring with $\mathbb{Q}$. Since $k$ is algebraically closed, $S_r(k; \underline{J_1}\times\ldots\times\underline{J_d}\times A)$ is divisible, so the top map is an isomorphism. After tensoring with $\Q$ the diagonal map becomes injective, since $\overline{\Psi}_r\circ \overline{\Phi}_r = r!$. Therefore, $q$ is injective after tensoring with $\Q$ and thus $F^{r+1}\ch_0(X)\otimes \mathbb{Q} = F^{r+1}_{\Psi}\ch_0(X) \otimes \mathbb{Q}$
\end{proof}

We also will require that our two filtrations $F^r\ch_0(X)$ and $F^r_{\Psi}\ch_0(X)$ behave well with respect to changes in the base field. If $L/k$ is an arbitrary extension, then we claim that the diagram 
\[
    \begin{tikzcd}
        \ch_0(X) \arrow{r}{\Phi_r} \arrow{d}{\pi_{L/k}^*} & S_r(k; \underline{J_1}\times \cdots \times \underline{J_d}\times A) \arrow{d}{\res_{L/k}} \\
        \ch_0(X_L) \arrow{r}{\Phi^L_r} & S_r(L; \left(\underline{J_1}\times \cdots \times \underline{J_d}\times A\right)_L)
    \end{tikzcd}
\]
commutes. Indeed if $z \in X$ is a closed point, then 
\begin{align*}
    \Phi_r^L(\pi_{L/k}^*([z])) &= \Phi_r^L\left(\sum_{w \in z\times_k \Spec L}[w]\right) \\
    &= \sum_{w \in z\times_k \Spec L}\{(\varphi_X\times L)(w), \ldots, (\varphi_X\times L)(w)\}_{k(w)/L} \\
    &= \sum_{w \in z\times_k \Spec L}\{\res_{k(w)/k(z)}(\varphi_X(z)), \ldots, \res_{k(w)/k(z)}(\varphi_X(z))\}_{k(w)/L} \\
    &= \res_{L/k}\{\varphi_X(z), \ldots, \varphi_X(z)\}_{k(z)/k}\\
    &= \res_{L/k}\Phi_r([z]).
\end{align*}
Therefore, $\pi_{L/k}^*$ respects the filtration $F^r\ch_0(X)$. 

Similarly, by the defintion of $\Psi'_r$ and the functoriality of pushforward, we get that if $L/k$ is a finite extension, then the diagram
\[
    \begin{tikzcd}
        S_r(L; \left(\underline{J_1}\times \ldots \times \underline{J_d}\times A\right)_L) \arrow{d}{\Tr_{L/k}} \arrow{r}{\Psi^{L'}_r} & \ch_0(X_L) \arrow{d}{(\pi_{L/k})_*} \\
        S_r(k; \underline{J_1}\times \cdots \times \underline{J_d}\times A) \arrow{r}{\Psi_r'} & \ch_0(X)
    \end{tikzcd}
\]
commutes. Therefore, $(\pi_{L/k})_*$ respects the filtration $F^r_{\Psi}\ch_0(X)$. 

Now we are ready to prove the proposition.
\begin{proof}
    Let $x \in F^r\ch_0(X)$. Then let $\overline{x} \in F^r\ch_0(X_{\overline{k}})$ be its image under the pullback. Since $F^r\ch_0(X_{\overline{k}})\otimes \Q = F^r_{\Psi}\ch_0(X_{\overline{k}})\otimes \Q$, we can write $x = \sum_i q_ix_i$ for some $x_i \in F^r_\Psi\ch_0(X_{\overline{k}})$. There exists some finite extension $L$ such that all the $x_i$ are defined over $L$ and $\res_{L/k}(x) = \sum_iq_ix_i$ in $\ch_0(X_L)$. Then $[L:k]x = \Tr_{L/k}\circ \res_{L/k}(x) = \sum_i q_i\Tr_{L/k}(x_i)$. Thus $x \in F_{\Psi}^r\ch_0(X)\otimes \Q$. 
\end{proof}
\section{Example: Chow Group of a Genus Two Curve and its Jacobian}\label{section: Example computation}
In this section, we show how the filtration can be used to study the Albanese kernel and produce rational equivalences in a special case. 

Suppose $k = \overline{k}$ is algebraically closed for simplicity. Let $C$ be a curve of genus $2$ over $k$ and $J$ be its Jacobian. Let the fixed point $p_1 \in C$ be a Weierstrass point and denote the embedding determined by $p_1$ by $\iota: C \to J$.  We want to study the group $\ch_0(X)$ where $X = C\times J$. Consider the filtration 
\[
    \ch_0(X) = F^0\ch_0(X) \supset F^1\ch_0(X) \supset F^2\ch_0(X)\supset F^3\ch_0(X) \supset \ldots 
\]
Since $k$ is algebraically closed, $F^2(X)$ is torsion free which was proven by Roitman \cite{rojtman1980torsion} for the torsion prime to $\text{char}(k)$ and proven for $\text{char}(k)$ torsion by Milne \cite{milne1982zero}. Since $F^4\ch_0(X) \otimes \Q = 0$ by the previous section, we get that $F^4\ch_0(X) = 0$. Therefore, by Theorem \ref{thm: main}, we have
\[
    \Z[1/3!]\otimes F^3\ch_0(X) \cong \mathbb{Z}\left[1/3!\right]\otimes S_3(k; \underline{J}\times J).
\]
This $K$-group is generated by symbols of the form
\[
    \{([y_1]-[p_1], 0), (0,[y_2]-[p_1]),(0,[y_3]-[p_1])\}
\]
and 
\[
    \{(0, [y_1]-[p_1]),(0,[y_2]-[p_1]),(0,[y_3]-[p_1])\}.
\]
Applying $\Psi_3$ to symbols of the second form, we get 
\[
    \Psi_3(\{(0, [y_1]-[p_1]),(0,[y_2]-[p_1]),(0,[y_3]-[p_1])\}) = ([a_1]-[0])\odot([a_2]-[0])\odot ([a_3]-[0])
\]
where $a_i = [y_i]-[p_i]$. However, as mentioned in Section \ref{section: finiteness of filt}, $([a_1]-[0])\odot([a_2]-[0])\odot ([a_3]-[0]) = 0$ in $\ch_0(J)$ since $\dim(J) = 2$. 

Therefore, the group $\Z[1/3!]\otimes F^3\ch_0(X)$ is generated by cycles of the form
\begin{align*}
    &[(y_1, \iota(y_2)+\iota(y_3))]-[(y_1, \iota(y_2))]-[(y_1, \iota(y_3))]+[(y_1, 0)] \\
    &-[(p_1, \iota(y_2)+\iota(y_3))]+[(p_1, \iota(y_2))]+[(p_1, \iota(y_3))]-[(p_1, 0)]
\end{align*}
where $y_1,y_2, y_3 \in C(k)$. We also claim that if $y_i = y_j$ for any $i \ne j$, then cycles of the above form are $0$ in the Chow group. To see this, it suffices to show that 
\[
    \{([y_1]-[p_1], 0), (0,[y_2]-[p_1]),(0,[y_3]-[p_1])\} = 0
\]
in the group $\mathbb{Z}\left[1/3!\right]\otimes S_3(k; \underline{J}\times J)$ if $y_i = y_j$ for $i\ne j$. Let us assume that $y_1 = y_2$. Since $p_1$ is a Weierstrass point by assumption, $[y_1]+[x]-2[p_1]$ is the divisor of a function $f \in k(C)^\times$ where $x$ is the image of $y_1$ under the hyperelliptic involution. We define three maps $g_1,g_2,g_3 : J\times J$ as follows. The map $g_1$ is $\iota\times 0$, $g_2$ is $0 \times \iota$, and $g_3$ is $0$ on the first component and the constant map to $[y_3]-[p_1]$ on the second component. Then Weil reciprocity with these three maps and the divisor $[y_1]+[x]-2[p_1]$ gives the following relation in $S_3(k; \underline{J}\times J)$.
\begin{align*}
    0 &= \{([y_1]-[p_1], 0), (0,[y_1]-[p_1]),(0,[y_3]-[p_1])\} + \{([x]-[p_1], 0), (0,[x]-[p_1]),(0,[y_3]-[p_1])\}\\
    &= \{([y_1]-[p_1], 0), (0,[y_1]-[p_1]),(0,[y_3]-[p_1])\} + \{-([y_1]-[p_1], 0), -(0,[y_1]-[p_1]),(0,[y_3]-[p_1])\} \\
    &= 2\{([y_1]-[p_1], 0), (0,[y_1]-[p_1]),(0,[y_3]-[p_1])\}. 
\end{align*}
But since we have inverted $2$ in $\mathbb{Z}\left[1/3!\right]\otimes S_3(k; \underline{J}\times J)$, $2\{([y_1]-[p_1], 0), (0,[y_1]-[p_1]),(0,[y_3]-[p_1])\} = 0$
implies $\{([y_1]-[p_1], 0), (0,[y_1]-[p_1]),(0,[y_3]-[p_1])\} = 0$.

Now let us consider the next quotient which is $F^2\ch_0(X)/F^3\ch_0(X)$. By Theorem \ref{thm: main}, we have that 
\[
    \Z\left[1/2\right]\otimes \frac{F^2\ch_0(X)}{F^3\ch_0(X)} \cong \Z\left[1/2\right]\otimes S_2(k; \underline{J}\times J).
\]
The $K$-group is generated by symbols of the form 
\[
    \{([y_1]-[p_1], 0),(0, [y_2]-[p_1])\}
\]
and 
\[
    \{(0, [y_1]-[p_1]),(0, [y_2]-[p_1])\}
\]
for $y_1,y_2 \in C(k)$. Thus, $\Z\left[1/2\right]\otimes F^2\ch_0(X)/F^3\ch_0(X)$ is generated by cycles of the form 
\[
    [(y_1,\iota(y_2))] - [(y_1, 0)]-[(p_1, \iota(y_2))]+[(p_1, 0)]
\]
and 
\[
    [(p_1,\iota(y_1)+\iota(y_2))] - [(p_1, \iota(y_1)]-[(p_1, \iota(y_2))]+[(p_1, 0)].
\]
A similar computation to the one performed above shows that for $y_1 \in C(k)$, 
\[
    [(y_1,\iota(y_1))] - [(y_1, 0)]-[(p_1, \iota(y_1))]+[(p_1, 0)] = 0
\]
in $F^2\ch_0(X)/F^3\ch_0(X)$ and in fact this cycle is trivial in $F^2\ch_0(X)$ which can be seen by looking at its image under the map $\Phi_2$. 

We also note that by the definition of $S_2(k;\underline{J}\times J)$ (Definition \ref{def: main K-group def}), we have that 
\[
    S_2(k;\underline{J}\times J) \cong K(k; J, J) \oplus S_2(k; J).
\]
Raskind and Spiess show in \cite{raskind2000milnor} that $K(k; J,J) \cong F^2\ch_0(C\times C)$ and Gazaki shows in \cite{gazaki2015filtration} that $\Z[1/2]\otimes S_2(k;J) \cong \Z[1/2]\otimes F^2\ch_0(J)$ where we take $F^2$ to denote the Albanese kernel. Therefore, the above shows that 
\[
    \Z[1/2]\otimes F^2\ch_0(X)/F^3\ch_0(X) \cong \Z[1/2]\otimes \left(F^2\ch_0(C\times C)\oplus F^2\ch_0(J)\right).
\]

This example demonstrates that the filtration allows us to break the group $F^2\ch_0(X)$ into smaller pieces and study each of them separately. The $K$-groups allow us to easily determine generators for the successive quotients. Also, the relations coming from Weil Reciprocity in the $K$-groups can help us to write down relations in the Albanese kernel $F^2\ch_0(X)$ coming from rational equivalence. 

In order to attempt to fully understand these successive quotients, considerations about the base field must be taken into account. For example, over $\C$ we know by the work of Mumford \cite{mumford1969rational}, we know that $F^2\ch_0(C\times C)$ and $F^2\ch_0(J)$ are both huge. So $F^2\ch_0(X)/F^3\ch_0(X)$ will be as well by the above. However, over $\Q$ conjectures of Beilinson and Bloch \cite{beilinson2006height} \cite{bloch1984algebraic} predict that $F^2\ch_0(X)$ should be $0$. It is easy to check by the above that if $F^2\ch_0(C\times C)$ and $F^2\ch_0(J)$ are $0$, then so is $F^2\ch_0(X)$, but neither case is known. Gazaki and Love have studied the Albanese kernel for Abelian surfaces over $\overline{\Q}$ in \cite{gazaki2023hyperelliptic}. In particular, they are able to produce many rational equivalences coming from hyperelliptic curves mapping to the abelian surface when it is isomorphic to a product of elliptic curves $E\times E'$.

\bibliography{biblio}

@article{gazaki2015filtration,
  title={On a filtration of for an abelian variety},
  author={Gazaki, Evangelia},
  journal={Compositio Mathematica},
  volume={151},
  number={3},
  pages={435--460},
  year={2015},
  publisher={London Mathematical Society}
}

@article{somekawa1990milnor,
  title={On Milnor K-groups attached to semi-abelian varieties},
  author={Somekawa, Mutsuro},
  journal={K-theory},
  volume={4},
  number={2},
  pages={105--119},
  year={1990},
  publisher={Springer Netherlands}
}

@article{gazaki2022filtrations,
  title={Filtrations of the Chow group of zero-cycles on abelian varieties and behavior under isogeny},
  author={Gazaki, Evangelia},
  journal={arXiv preprint arXiv:2210.14372},
  year={2022}
}

@book{fulton2013intersection,
  title={Intersection theory},
  author={Fulton, William},
  volume={2},
  year={2013},
  publisher={Springer Science \& Business Media}
}

@article{bloch1975k2,
  title={K2 of Artinian Q-algebras, with application to algebraic cycles},
  author={Bloch, Spencer},
  journal={Communications in Algebra},
  volume={3},
  number={5},
  pages={405--428},
  year={1975},
  publisher={Taylor \& Francis}
}

@article{mumford1969rational,
  title={Rational equivalence of 0-cycles on surfaces},
  author={Mumford, David},
  journal={Journal of mathematics of Kyoto University},
  volume={9},
  number={2},
  pages={195--204},
  year={1969},
  publisher={Duke University Press}
}

@inproceedings{beilinson2006height,
  title={Height pairing between algebraic cycles},
  author={Beilinson, Alexander A},
  booktitle={K-Theory, Arithmetic and Geometry: Seminar, Moscow University, 1984--1986},
  pages={1--26},
  year={2006},
  organization={Springer}
}

@article{bloch1984algebraic,
  title={Algebraic cycles and values of {L}-functions.},
  author={Bloch, Spencer},
  year={1984},
  publisher={Walter de Gruyter, Berlin/New York Berlin, New York}
}

@article{raskind2000milnor,
  title={Milnor K-groups and zero-cycles on products of curves over p-adic fields},
  author={Raskind, Wayne and Spiess, Michael},
  journal={Compositio Mathematica},
  volume={121},
  number={1},
  pages={1--34},
  year={2000},
  publisher={London Mathematical Society}
}

@book{bloch2010lectures,
  title={Lectures on algebraic cycles},
  author={Bloch, Spencer},
  volume={16},
  year={2010},
  publisher={Cambridge University Press}
}

@article{bloch1976some,
  title={Some elementary theorems about algebraic cycles on abelian varieties},
  author={Bloch, Spencer},
  journal={Inventiones mathematicae},
  volume={37},
  number={3},
  pages={215--228},
  year={1976},
  publisher={Springer}
}

@article{beauville1986anneau,
  title={Sur l'anneau de Chow d'une vari{\'e}t{\'e} ab{\'e}lienne},
  author={Beauville, Arnaud},
  journal={Mathematische Annalen},
  volume={273},
  number={4},
  pages={647--651},
  year={1986},
  publisher={Springer}
}

@article{murre1991motivic,
  title={Motivic decomposition of abelian schemes and the Fourier transform.},
  author={Murre, Jacob and Deninger, Ch},
  year={1991},
  publisher={Walter de Gruyter, Berlin/New York Berlin, New York}
}

@article{kahn2013voevodsky,
  title={Voevodsky’s motives and Weil reciprocity},
  author={Kahn, Bruno and Yamazaki, Takao},
  year={2013}
}

@article{kakinoki2020filtration,
  title={A filtration on the higher Chow group of zero cycles on an abelian variety},
  author={Kakinoki, Buntaro},
  year={2020}
}

@article{rojtman1980torsion,
  title={The torsion of the group of 0-cycles modulo rational equivalence},
  author={Rojtman, AA},
  journal={Annals of Mathematics},
  volume={111},
  number={3},
  pages={553--569},
  year={1980},
  publisher={JSTOR}
}

@article{milne1982zero,
  title={Zero cycles on algebraic varieties in nonzero characteristic: Rojtman's theorem},
  author={Milne, James S},
  journal={Compositio mathematica},
  volume={47},
  number={3},
  pages={271--287},
  year={1982}
}

@article{gazaki2023hyperelliptic,
  title={Hyperelliptic curves mapping to abelian varieties and applications to Beilinson's conjecture for zero-cycles},
  author={Gazaki, Evangelia and Love, Jonathan R},
  journal={arXiv preprint arXiv:2309.06361},
  year={2023}
}

@article{jannsen1994motivic,
  title={Motivic sheaves and filtrations on Chow groups},
  author={Jannsen, Uwe},
  year={1994},
  publisher={American Mathematical Society}
}

@article{bloch1986algebraic,
  title={Algebraic cycles and higher K-theory},
  author={Bloch, Spencer},
  journal={Advances in mathematics},
  volume={61},
  number={3},
  pages={267--304},
  year={1986},
  publisher={Academic Press}
}
\bibliographystyle{alpha}

\end{document}